\documentclass{amsart}

\usepackage[colorlinks=True, citecolor=blue, linkcolor=blue]{hyperref}
\usepackage{amssymb}
\usepackage{tikz-cd}
\usepackage{tkz-graph}
\usepackage{algorithm}
\usepackage{algorithmic}
\usepackage[nameinlink]{cleveref}
\usepackage{booktabs}
\usepackage{longtable}
\usepackage{placeins}
\usepackage{colortbl}
\usepackage{color}
\usepackage[shortlabels,inline]{enumitem}

\usepackage{comment}
\definecolor{lightgray}{gray}{0.9}

\theoremstyle{definition}
\newtheorem{definition}{Definition}[section]
\newtheorem{example}[definition]{Example}
\newtheorem{remark}[definition]{Remark}

\theoremstyle{plain}
\newtheorem{theorem}[definition]{Theorem}
\newtheorem{lemma}[definition]{Lemma}
\newtheorem{proposition}[definition]{Proposition}
\newtheorem{corollary}[definition]{Corollary}

\DeclareMathOperator{\Nef}{Nef}

\DeclareMathOperator{\rk}{rk}

\DeclareMathOperator{\id}{id}

\DeclareMathOperator{\Aut}{Aut}

\DeclareMathOperator{\aut}{aut}

\DeclareMathOperator{\im}{im}

\DeclareMathOperator{\stab}{stab}

\newcommand{\QQ}{\mathbb{Q}}
\newcommand{\FF}{\mathbb{F}}
\newcommand{\RR}{\mathbb{R}}
\newcommand{\CC}{\mathbb{C}}
\newcommand{\ZZ}{\mathbb{Z}}
\newcommand{\PP}{\mathbb{P}}

\newcommand{\HH}{\mathbb{H}}

\newcommand{\C}{\mathcal{C}}
\newcommand{\D}{\mathcal{D}}
\newcommand{\E}{\mathcal{E}}
\newcommand{\F}{\mathcal{F}}

\newcommand{\M}{\mathcal{M}}

\renewcommand{\O}{\mathcal{O}}
\renewcommand{\P}{\mathcal{P}}

\newcommand{\R}{\mathcal{R}}

\newcommand{\Num}{\textrm{Num}}

\newcommand{\RN}[1]{\textup{\uppercase\expandafter{\romannumeral#1}}}

\makeatletter
\newcommand*{\defeq}{\mathrel{\rlap{%
                     \raisebox{0.3ex}{$\m@th\cdot$}}%
                     \raisebox{-0.3ex}{$\m@th\cdot$}}%
                     =}
\makeatother

\newcommand{\Gbar}{\overline{G}}
\newcommand{\Pbar}{\overline{\P}^\mathbb{Q}}
\newcommand{\taubar}{\bar{\tau}}

\makeatletter
\def\blfootnote{\xdef\@thefnmark{}\@footnotetext}

\title[]{527 Elliptic Fibrations on Enriques surfaces}

\author{Simon Brandhorst}

\address{Simon Brandhorst,
Fakult\"at f\"ur Mathematik und Informatik, Universit\"at des Saarlandes, Campus E2.4, 66123 Saarbr\"ucken, Germany}
\email{brandhorst@math.uni-sb.de}
\author{Víctor González-Alonso}
\address{Víctor González-Alonso,
Institut für Algebraische Geometrie,
Leibniz Universit\"at Hannover
Welfengarten 1,
30167 Hannover, Germany}
\email{gonzalez@math.uni-hannover.de}
\thanks{Gefördert durch die Deutsche Forschungsgemeinschaft (DFG) – Projektnummer 286237555 – TRR 195.
Funded by the Deutsche Forschungsgemeinschaft (DFG, German Research Foundation) – Project-ID 286237555 – TRR 195.}
\begin{document}
\begin{abstract}
Barth and Peters showed that a general complex Enriques surface has exactly 527 isomorphism classes of elliptic fibrations.
We show that every Enriques surface has precisely 527 isomorphism classes of elliptic fibrations when counted with the appropriate multiplicity. Their reducible singular fibers and the multiplicities can be calculated explicitly.
The same statements hold over any algebraically closed field of characteristic not two.
To explain these results, we construct a moduli space of complex elliptic Enriques surfaces
and study the ramification behavior of the forgetful map to the moduli space of unpolarized Enriques surfaces.
Curiously, the ramification indices of a similar map compute the hyperbolic volume of the rational polyhedral
fundamental domain appearing in the Morrison-Kawamata cone conjecture.
\end{abstract}
\maketitle
\addtocontents{toc}{\protect\setcounter{tocdepth}{1}}
\section{Introduction}
We work over an algebraically closed field $k = \overline{k}$ of characteristic not two.
An Enriques surface is a smooth proper surface $Y/k$ of Kodaira dimension $\kappa(Y)=0$ and second  Betti number $b_2(Y)=10$.  They are one of the four classes of Kodaira dimension zero in the Enriques-Kodaira classification of minimal algebraic surfaces.

An Enriques surface is called unnodal if it does not contain smooth rational curves. A general Enriques surface is unnodal.
In \cite{barth-peters:very_general_enriques} Barth and Peters show that a general unnodal complex Enriques surface $S$ has exactly $527$ elliptic fibrations up to isomorphism.  The result was extended to positive characteristic by Martin \cite{martin:automorphisms_of_unnodal_enriques_surfaces}.
If $Y$ is a general complex $1$-nodal Enriques surface, then Cossec and Dolgachev \cite{cossec-dolgachev:automorphisms_nodal} announced that
$Y$ has $255$ isomorphism classes of elliptic fibrations
with $\tilde{A_1}$ reducible fibers and $136$ classes with no reducible fibers (see e.g. \cite[Section 8.4, Table 8.4]{enriquesII} \cite[6.5]{brandhorst_shimada:tautaubar} for proofs).
This work is prompted by an idea of Cossec communicated to us by Dolgachev: The formula $527=255+2\cdot 136$ should have an explanation in terms of the ramification behavior of a forgetful map of moduli spaces.

More precisely, an elliptic fibration on an Enriques surface $Y$ is a morphism $\phi \colon Y \to C$ to a smooth curve $C$ such that its generic fiber is a smooth curve of genus one over the function field $k(C)$. Since $Y$ is an Enriques surface, it holds $C=\mathbb{P}^1$. The numerical class $\left[\phi^{-1}\left(p\right)\right]\in\Num\left(Y\right)$ is $2$-divisible, and  $f:=\frac{1}{2} [\phi^{-1}(p)] \in\Num\left(Y\right)$ is the so called \emph{half fiber}.
Two elliptic fibrations $\phi,\phi'\colon Y\to\mathbb{P}^1$ are isomorphic if there are $\psi \in \Aut(Y)$ and $\eta \in \Aut(\PP^1)$ with $\phi\circ \psi = \eta \circ \varphi'$, which holds if and only if the corresponding half fibers $f$ and $f'$ lie in the same $\Aut(Y)$-orbit.

In \Cref{sect:moduli-general} we construct the moduli space $\M_{En,e}$ of complex elliptic Enriques surfaces $(Y,f)$ and show that the forgetful map
\[\P^e_0\colon \M_{En,e} \to \M_{En}, \quad [Y,f] \mapsto [Y]\]
is a ramified cover of normal quasi-projective varieties. The fiber over $[Y]$ counts isomorphism classes of elliptic fibrations on $Y$. By the result of Barth and Peters the morphism $\P^e_0$ has degree $527$. This leads us to the following theorem.

\begin{theorem}\label{thm:intro}
Let $k = \mathbb{C}$ and $Y/k$ be an Enriques surface. Then $Y$ has exactly $527$ isomorphism classes of elliptic fibrations $\varphi$ counted with the ramification degree
$r_f$ of $\P^e_0$ at $[Y,f]$.
\end{theorem}

The ramification degrees are calculated as follows:
Let $S_Y:=\Num(Y)$ denote the numerical lattice of $Y$, $\aut(Y)$ the image of the natural homomorphism $\Aut(Y) \to O(S_Y)$, and $W(Y)\leq O(S_Y)$ the nodal Weyl group of $Y$, i.e. the subgroup generated by the reflections in smooth rational curves of $Y$.
Set $G_Y := W(Y) \rtimes \aut(Y)$ and define $\Gbar_Y$ as its image in $O(S_Y \otimes \FF_2)$. We call $\Gbar_Y$ the Vinberg group.
For any class $f \in S_Y$ denote by $\overline{f}$ its image in $S_Y \otimes \FF_2$.
\begin{proposition}\label{prop:intro} Let $Y$ be a complex Enriques surface and $f\in S_Y$ a primitive, isotropic nef class. Then the following hold.
\begin{enumerate}
\item[(a)]  The ramification degree $r_f$ of $\P^e_0$ at $[Y,f]$ is the orbit length
\begin{equation}\label{eq:def_ep}
    r_f=\#\left(\overline f\cdot \Gbar_Y \right).
\end{equation}
\item[(b)] The set of isomorphism classes of elliptic fibrations on $Y$ is in bijection with the quotient
\begin{equation} \label{eq:isom_ell_fibr_as_quotient}
    \{\overline{f} \in S_Y\otimes \FF_2 \mid \overline{f}^2=0, \overline{f}\neq 0\}/\Gbar_Y.
\end{equation}
\end{enumerate}
\end{proposition}
See \Cref{prop:ADE-type} for how to compute the ADE-type of the singular fibers of the elliptic fibration corresponding to $f$ from $\overline{f}$ and the so called spliting roots $\overline{\Delta}(Y)$ of $Y$.
\begin{remark}
In positive characteristic the existence and geometry of a moduli space of elliptic Enriques surfaces is not clear to us. However, the formula for the ramification degree $r_f$ makes perfect sense in any characteristic. If we take the formula in \cref{eq:def_ep} as the \emph{definition} of $r_f$, then \Cref{thm:intro} and \Cref{prop:intro} (2) continue to hold. See \Cref{thm:elliptic_fibrations} for the precise statement.
\end{remark}
The main idea for a proof in arbitrary characteristic is
\Cref{thm:NefmodAut}, which states that
$\Nef^e_Y/\Aut(Y) \cong \Pbar_Y/G_Y$ where $\Pbar_Y$ is the positive cone of $Y$ together with the rational isotropic rays. While $\Nef^e_Y$ and $\Aut(Y)$ are notoriously difficult to control, $G_Y$ and the positive cone can be handled, even in positive characteristic.
Indeed, $G_Y$ always contains the $2$-congruence subgroup of $O(S_Y)$ (\Cref{prop:G0}) and therefore it suffices to know its image $\Gbar_Y$ in $O(S_Y \otimes \FF_2)$.

\begin{remark}
We call $\Gbar_Y$ the Vinberg group because its index in $O(S_Y \otimes \FF_2)$ is the number of Vinberg chambers contained in the fundamental domain of the action of $\Aut(Y)$ on the effective nef cone of $Y$ appearing in the Morrison-Kawamata-Cone conjecture. See \Cref{corollary:nefmodaut}.
\end{remark}
In \cite{brandhorst_shimada:tautaubar} the first author and Shimada computed $\Aut(Y)$-orbits of smooth rational curves and elliptic fibrations on 182 families of so called $(\tau,\taubar)$-generic Enriques surfaces with computer aided methods. In \Cref{sect:tau-taubar} we give a conceptual proof of their computational results.

More generally, let $h \in E_{10}\cong U\oplus E_8$ with $h^2\geq 0$ and denote by $\M_{En,h}$ the moduli space of complex numerically $h$-quasi-polarized Enriques surfaces. It parametrises pairs consisting of an Enriques surface $Y$ and a nef numerical class $H \in S_Y$ mapping to $h$ under some marking $S_Y \xrightarrow{\sim} E_{10}$. For $h^2=0$, we obtain the moduli space of elliptic Enriques surfaces and for $h=0$ the moduli space of unpolarized Enriques surfaces.
In \cite{gritsenko-hulek} Gritsenko and Hulek show that there are only finitely many moduli spaces $\M_{En,h}$ up to isomorphism although there are infinitely many types of polarisations $h \in E_{10}$.
In \cite{hulek-dutour} Hulek and Dutour-Sikirić computed them along with the degrees of the forgetful maps $\P^h_0\colon \M_{En,h} \mapsto \M_{En,0}$.
Further detailed studies of moduli spaces of polarized Enriques surfaces were carried out in \cite{ciliberto23} and \cite{knutsen}.

Let $v \in E_{10}$ be a so called Weyl vector (see \Cref{sec:vinberg}) and $h \in E_{10}$ in the corresponding Vinberg chamber of $E_{10}$. Then there are finite proper surjective morphisms of normal quasi-projective complex varieties
\[\M_{En,v} \xrightarrow{\P^v_h}\M_{En,h} \xrightarrow{\P^h_0}{\to} \M_{En,0}.\]
The first one is Galois, but depends on the choice of $h$, and the second one is the forgetful map and in general not Galois.
The composition is just $\P^v_0$, hence Galois too. In \Cref{thm:branch} we use the sandwich to compute the ramification indices of $\P^h_0$. In particular, we answer a question posed by Dolgachev and Kondo in their upcoming book \cite[Table 8.3]{enriquesII} on counts of polarisations on general 1-nodal Enriques surfaces.

\subsection*{Plan of the paper}
In \Cref{sect:preliminaries} we set up terminology for lattices and recall Vinbergs results on the $E_{10}$-root lattice. \Cref{sect:nef-cone} recalls the description of the nef cone of an Enriques surface.
\Cref{sect:cone-conj} deals with the action of the automorphism group on the effective nef cone.
In \Cref{sect:orbits-fibrations} we look at orbits of elliptic fibrations under the automorphism group.
In \Cref{sect:tau-taubar} the results are applied to $(\tau,\taubar)$-generic Enriques surfaces.
In \Cref{sect:moduli-general} we review and extend the construction of numerically quasi-polarized complex Enriques surfaces.
In \Cref{sect:moduli-fibrations} we reinterpret the results in terms of ramification of maps between moduli spaces of Enriques surfaces.

\subsection*{Acknowledgements}
The authors would like to express their gratitude to Igor Dolgachev for asking the right questions and hinting at their answers.
We thank Gebhard Martin, Giacomo Mezzedimi, Klaus Hulek, Matthias Schütt, Ichiro Shimada and Davide Veniani for discussions.

\section{Preliminaries and Notation.}\label{sect:preliminaries}
See \cite{nikulin} for background on lattices and their discriminant groups.
A \emph{lattice} is a finitely generated $\ZZ$-module $L$ equipped with a non-degenerate, $\QQ$-valued symmetric bilinear form denoted by $x.y$.
For $n \in \ZZ$ we denote by $L(n)$ the lattice with the same underlying module but with bilinear form multiplied by $n$.
We denote the \emph{dual lattice} of $L$ by $L^\vee$.
The finite group $L^\vee/L$ is called the discriminant group. We call a vector $r \in L$ of square $r^2=-2$ a \emph{root} of $L$. It defines a reflection $s_r(x)=x+(r.x)r$.
A lattice of signature $(1,*)$ is called hyperbolic. The unique even unimodular lattice of signature $(1,1)$ is called a hyperbolic plane and denoted by $U$.
We follow the convention that $ADE$-root lattices are negative definite.
The orthogonal group of $L$ is denoted by $O(L)$.
We denote by $O^\sharp(L)$ the kernel of $O(L)\to O(L^\vee/L)$.
For subsets $A_1, \dots, A_n \subseteq L \otimes \RR$ we denote by $O(L,A_1, \dots , A_n)=\{f \in O(L) \mid \forall i \in \{1, \dots n\}\colon f(A_i) = A_i\}$ their joint stabiliser. For $L$ a hyperbolic lattice,
$\{x \in L \otimes \RR \mid x^2>0\}$ has two connected components. We fix one of them and call it the positive cone $\P_L$ of $L$.
Let $\Delta \subseteq \{x \in L \mid x^2 <0 \}$ be a subset. A $\Delta$-\emph{chamber} is a connected component of $\mathcal{P}_L \setminus \bigcup_{r \in \Delta} r^\perp$. Any partition
$\Delta = \Delta^+ \sqcup - \Delta^+$ determines a unique chamber $\{ x \in \P_L \mid \forall r \in \Delta^+: x.r>0\}$, which we call \emph{the} chamber defined by $\Delta^+$. We call the elements of $\Delta^+$ positive.

Let $G$ be a group acting on a set $X$.
We denote right actions by $x^g$ for $x \in X$ and $g \in G$. Left actions are denoted $gx$.

\subsection{The Vinberg lattice}\label{sec:vinberg}
\begin{figure}
\begin{tikzpicture}
\tikzset{VertexStyle/.style= {fill=black, inner sep=1.5pt, shape=circle}}
\Vertex[NoLabel,x=0,y=0]{1a}
\Vertex[NoLabel,x=1,y=0]{2a}
\Vertex[NoLabel,x=2,y=0]{3a}
\Vertex[NoLabel,x=3,y=0]{4a}
\Vertex[NoLabel,x=4,y=0]{5a}
\Vertex[NoLabel,x=5,y=0]{6a}
\Vertex[NoLabel,x=6,y=0]{7a}
\Vertex[NoLabel,x=7,y=0]{8a}
\Vertex[NoLabel,x=8,y=0]{9a}
\Vertex[NoLabel,x=6,y=1]{10a}
\Edges(1a,2a,3a,4a,5a,6a,7a,8a,9a)
\Edges(7a,10a)
\tikzset{VertexStyle/.style= {inner sep=1.5pt, shape=circle}}
\Vertex[x=0,y=-0.4]{1}
\Vertex[x=1,y=-0.4]{2}
\Vertex[x=2,y=-0.4]{3}
\Vertex[x=3,y=-0.4]{4}
\Vertex[x=4,y=-0.4]{5}
\Vertex[x=5,y=-0.4]{6}
\Vertex[x=6,y=-0.4]{7}
\Vertex[x=7,y=-0.4]{8}
\Vertex[x=8,y=-0.4]{9}
\Vertex[x=6,y=1.4]{10}
\end{tikzpicture}
\caption{The $E_{10}$ diagram.}\label{figure:E10}
\end{figure}
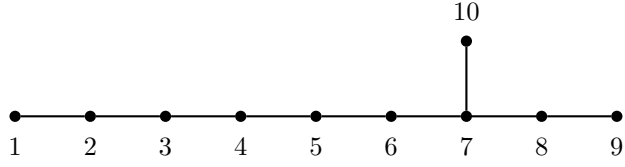
Let $E_{10}$ be the lattice corresponding to the Coxeter-Dynkin diagram in \Cref{figure:E10}.
It is an even unimodular lattice of signature $(1,9)$ and therefore $E_{10}\cong U \oplus E_8$.
Note that if $Y$ is an Enriques surface, then $S_Y\cong E_{10}$.
Let $e_{1}, \dots e_{10}$ be the basis of $E_{10}$ corresponding to its Coxeter-Dynkin diagram.
In \cite{vinberg}, Vinberg showed that $e_{1}, \dots , e_{10}$ constitute a fundamental system of roots of $E_{10}$. In particular,
$W(E_{10}) = \langle s_e | e \in \{e_1,\dots e_{10}\}\rangle$ and moreover $O(E_{10}) = \{\pm 1\} \times W(E_{10})$.
Since $E_{10}$ is unimodular, we find a unique $v \in E_{10}$ with $v.e_1=\dots=v.e_n = 1$. We call it the
\emph{Weyl vector} of the \emph{Vinberg chamber}
\[C = \{ x \in E_{10}\otimes \RR \mid \forall i \in {1,\dots 10}\colon x.e_i \geq 0\}.\]
Since the $e_i$ are a basis of the $E_{10}$ lattice, this convex cone is simplicial and moreover basic.
Indeed, its extremal rays are given by the dual basis $e_1^\vee, \dots e_n^\vee \in E_{10}$ of $e_1,\dots e_n$ where $e_i^\vee.e_j = 1$ for $i=j$ and $0$ else.
We see that $v = e_1^\vee + \dots + e_{10}^{\vee}$ is an interior point and one can calculate that $v^2= 1240$.
Since the chamber $C$ is uniquely determined by $v$, we write $C(v)=C$.  It is a fundamental domain for the action of the Weyl group on the positive cone because the $e_i$ for a fundamental system of roots. We call any integral $v'$ in the positive cone of \(E_{10}\) a Weyl vector if the set $\{x \in E_{10} \mid x^2 = -2, x.v'=1\}$ is a fundamental root system. This is the case if and only if they form an $E_{10}$ diagram.
This sets up bijections between fundamental root systems, chambers and Weyl vectors. The Weyl group acts simply transitively on each of these sets in a compatible way.

We equip $E_{10} \otimes \FF_2$ with a non-degenerate $\FF_2$-valued quadratic form $q$ defined by $q(x)= x^2/2 \mod 2$.
It can be interpreted as the discriminant form of $E_{10}(2)$ by identifying $\tfrac{1}{2}E_{10}/E_{10}$ with $E_{10}\otimes \FF_2$. The quadratic space $E_{10} \otimes \FF_2$ has a totally isotropic subspace of dimension $5$.

\section{The nef cone.} \label{sect:nef-cone}
In this section we recall the description of the nef cone of an Enriques surface in the form that suits our purposes. For the reader's convenience we give most proofs.

In what follows $Y$ denotes an Enriques surface over an algebraically closed field of characteristic $p\neq 2$ with covering K3 surface $X$, covering map $\pi\colon X \to Y$ and covering involution $\epsilon \in \Aut(X)$.
Denote by $S_Y\cong E_{10}\cong U \oplus E_8$ and $S_X$ the numerical lattices of $Y$ respectively $X$.

\subsection{The ample cone.}
We shall derive the description of the ample cone of the Enriques surface $Y$ from its covering K3 surface $X$.
\begin{lemma}\label{lem:ampleY}
A divisor class $h \in S_Y$ is ample on $Y$ if and only if $\pi^*(h)$ is ample on $X$.
\end{lemma}
\begin{proof}
Let $h \in S_Y$ be ample and $C$ be any curve on $X$. Then $\pi_*(C)$ is a curve on $Y$ and
by the projection formula $\pi^*(h).C = h.\pi_*(C)>0$, which is positive because $h$ is ample.
Thus $\pi^*(h)$ is ample by the Nakai-Moishezon criterion.
Conversely if $h \in S_Y$ is such that $\pi^*(h)$ is ample and $C$ is any curve on $Y$,
then $\pi^*(C)$ is a curve on $Y$ and therefore $h.C = \frac{1}{2} \pi^*(h),\pi^*(C)>0$. Thus $h$ is ample.
\end{proof}

Since the lattice $S_Y$ is hyperbolic, the cone $\{ x \in S_Y \mid x^2>0\}$ has two connected components.
The one containing the ample cone is called the \emph{positive cone} $\mathcal{P}_Y$.

We want to describe the ample cone as a subset of the positive cone $\mathcal{P}_Y$. Let $C \subseteq Y$ be a smooth rational curve.
Then $C^2=-2$ and the preimage $\pi^{-1}(C)$ in $X$ splits into two disjoint smooth rational curves $\tilde{C}$ and $\epsilon(\tilde{C})$. This motivates the following definition.

\begin{definition}
We say that a root $r \in S_Y$ \emph{splits} in $S_X$ if there exists a root $\tilde{r}$ of $S_X$ such that $\pi^*(r)= \tilde{r} + \epsilon^*(\tilde{r})$. 
We denote the set of splitting roots by $\Delta(Y)\subseteq S_Y$
and the subset of effective splitting roots by $\Delta(Y)^+$.
\end{definition}
The lattice $S_+:=\pi^*(S_Y)\cong S_Y(2)$ is a primitive sublattice of $S_X$, and we set $S_{-} := S_{+}^\perp \subseteq S_X$, which is also a primitive sublattice. Note that $S_\pm = \{x \in S_X \mid \epsilon^*(x) = \pm x\}$.
Since $S_+$ contains an ample class, its signature is $(1,*)$ and hence $S_-$ is negative definite. Moreover, $S_-$ does not contain any roots. Indeed, if $r$ is a root in $S_-$, then $r$ or $-r$ is effective by Riemann-Roch. Choose $a \in S_Y$ ample. Then $\pi^*(a).r=0$, contradicting that $\pi^*(a)$ is ample.
\begin{lemma}\label{lem:splitting-root}
A root $r \in S_Y$ splits if and only if
there exists $v \in S_{-}$ with $ v^2=-4$ and  $(\pi^*(r)+v)/2 \in S_X$.
\end{lemma}
\begin{proof}
Let $r \in S_Y$ be a splitting root. Then there exists a root $\tilde{r} \in S_X$ with $\pi^*(r) = \tilde{r}+\epsilon^*(\tilde{r}) \in S_+$. 
Set $v = \tilde{r} - \epsilon^*(\tilde{r}) \in S_{-}$. Then $(\pi^*(r) +v)/2 = \tilde{r} \in S_X$ and $-2 = \tilde{r}^2 =\frac{1}{4}\left(\pi^*(r)^2 + v^2\right) = -1 + \frac{1}{4}v^2$ gives $v^2 = -4$.

Conversely, let $r \in S_Y$ be a root. If there exists a $v \in S_{-}$ with $v^2=-4$ and $\tilde{r}:=(\pi^*(r)+v)/2 \in S_X$, then $\tilde{r}^2=\frac{1}{4}\left(\pi^*(r)^2+v^2\right)=-2$ and $\tilde{r} + \epsilon^*(\tilde{r}) =\pi^*(r)$. Thus $r$ splits in $S_X$.
\end{proof}
\begin{lemma}
The ample cone of $Y$ is the $\Delta(Y)$-chamber defined by $\Delta(Y)^+$.
\end{lemma}
\begin{proof}
Let $\Delta(X)$ be the set of roots of $S_X$ and \(\Delta^+(X)\) its subset consisting of effective roots.
It is known \cite[VIII (3.9)]{bhpv:compact-complex-surfaces} that the ample cone of $X$ is the $\Delta(X)$-chamber defined by $\Delta^+(X)$ .
By \Cref{lem:ampleY} the ample cone of $Y$ can be seen as the intersection of the ample cone of $X$ with $S_+\otimes \RR$.

We calculate which root hyperplanes $r^\perp\subseteq S_X\otimes\RR$ actually meet $\pi^*(\P_Y)$.
Let $h \in \P_Y$ and $r \in \Delta^+(X)$ with $\pi^*(h).r=0$. Since $h' := \pi^*(h) \in S_+$, $h'.r = h'.r_+$, where
$r_\pm \in S_\pm^\vee$ denotes the corresponding orthogonal projection of $r$.
Then $r_+$ lies in the orthogonal complement $h'^\perp \cap S_+^\vee$ of $h'$. Since it is negative definite, $r_+^2 \leq 0$. If $r_+^2=0$, then $r=r_-$ is a root in $S_X \cap S_-^\vee = S_-$ which is absurd.
Now $0 > r_+^2 \geq r_+^2+r_-^2= r^2 = -2$.
If $r_+^2=-2$, then $r_-^2=0$; hence $r=r_+ \in S_+^\vee\cap S_X = S_+ \cong E_{10}(2)$. But since $E_{10}$ is an even lattice, any class $x$ in $E_{10}(2)$ satisfies $4 \mid x^2$, which contradicts $r^2=-2$. Thus $r_+^2>-2$.
Since $2r_+ \in S_+\cong E_{10}(2)$,  $(2r_+)^2$ is divisible by $4$. Thus $r_+^2=-1$
and $r_-^2 =-2 - r_+^2= -1$. This shows that the preimage $(\pi^*)^{-1}(r_+) \in \Delta^+(Y)$ is a splitting root.
\end{proof}

\subsection{The effective nef  cone}
Recall that an $\RR$-divisor is called \emph{nef} if it has non-negative intersection with every curve. The set of nef divisors is the closure of the ample cone.
To assure that the action of the automorphism group has a rational polyhedral fundamental domain,
we replace the nef cone $\Nef(Y)$ by its effective version:
the \emph{effective nef cone} is the intersection of the nef cone with the effective cone of $Y$.
It is the convex cone spanned by the intersection of the nef cone and $S_Y$.
We let $h\in S_Y$ be ample and
\[\Pbar_Y = \P_Y \cup \RR_+ \{r \in L \otimes \QQ \mid r^2 = 0, h.r>0\}\]
be the convex cone given by $\P_Y$ and the rays of rational classes in its boundary (excluding the apex $0$).
We denote by $\Nef^{e}_Y$ the closure of the ample cone in $\Pbar_Y$,
which is the effective nef cone with the origin removed. By abuse of terminology we call it the effective nef cone as well.
Inside $\P_Y$, the nef cone and its effective version agree, i.e.
$\Nef(Y) \cap \P_Y = \Nef^e_Y \cap \P_Y$. In particular, they have the same facets. But they may have different extremal rays.

Let $\Delta\subseteq S_Y$ be a set of roots. For a root $r \in \Delta$, recall that $s_r(x) = x + (x.r)r$ defines a reflection into the hyperplane $r^\perp$. The \emph{Weyl}-group $W(\Delta)=\langle s_r \mid r \in \Delta \rangle$ is the subgroup of $O(S_Y)$ generated by all such reflections.
Let $\aut(Y)$ denote the image of the natural map $\psi\colon \Aut(Y) \to O(S_Y)$.
The kernel of $\psi$ is finite of order at most $4$ \cite{mukai_namikawa:numerically_trivial,dolgachev_martin:numericallytrivial}.

Recall that for a subset $A \subseteq S_Y\otimes \RR$, we denote by $O(S_Y,A)$ the set of isometries preserving $A$. Define  $O(S_Y,\P_Y, \pi)$ as the image of $O(S_X,S_+,\P_X) \to O(S_Y,\P_Y)$.

\begin{remark}\label{remark:split_weyl}
Let $r \in \Delta(Y)$ be a split root. Then $\pi^*r=\tilde{r} + \epsilon^*(\tilde{r})$ for some root $\tilde{r} \in S_X$.
We remark that $s_{\tilde{r}}\circ s_{\epsilon^*(\tilde{r})}$ commutes with $\epsilon^*$ and agrees with the reflection $s_{r}$ on $S_Y(2) \cong S_+$.
\end{remark}

For an Enriques surface $Y$ we define its Weyl group as $W(Y)=W(\Delta(Y))$.
It is a subgroup of $O(S_Y,\P_Y, \pi)$ by \Cref{remark:split_weyl}. It is moreover normal because the definition of $\Delta(Y)$ depends only on $\pi^*\colon S_Y \to  S_X$.
\begin{definition}
The \emph{modular stabiliser} of $Y$ is defined as
\[G_Y := W(Y) \rtimes \aut(Y) \subseteq O(S_Y, \P_Y, \pi).\]
\end{definition}
The modular stabiliser $G_Y$ is our key player. It will turn out to be much easier to control than $\aut(Y)$ or $W(Y)$.
We now characterise $G_Y$ in terms of the K3 cover $\pi\colon X \to Y$.
To this end let $\Aut(X,\epsilon)$ be the centraliser of $\epsilon$ in $\Aut(X)$ and $\aut(X,\epsilon)$ its image in $O(S_X)$. Note that $\Aut(Y) \cong \Aut(X,\epsilon)/\langle \epsilon \rangle$.
Let $W(X)\subseteq O(S_X)$ be the subgroup generated by the reflections on roots of $X$, $W(X,\epsilon)$ the centraliser of $\epsilon^*$ in $W(X)$, and set
\begin{equation}\label{def:GX}
G_{X|Y} \defeq W(X,\epsilon)\rtimes \aut(X,\epsilon) \subseteq O(S_X).
\end{equation}
We see that $W(X,\epsilon)$ restricts to $W(Y)$ under the identification of $S_+$ with $S_Y$.

Let $G_\pm$ be the image of $G_{X|Y}$ in $O(S_{X\pm})$.
\begin{proposition}\label{prop:GYG_+}
Let $Y$ be an Enriques surface. Then $G_+ \cong G_Y$ under the natural identification $O(S_Y)\cong O(S_{+})$.
\end{proposition}
\begin{proof}
The proof of \cite[Proposition 3.2]{brandhorst_shimada:tautaubar} works in positive characteristic as well.
\end{proof}

\begin{proposition}\label{prop:fundamental_domain}
Let $Y$ be an Enriques surface. Then:
\begin{enumerate}
    \item The effective nef cone $\Nef^{e}_Y$ is a fundamental domain for the action of the Weyl group $W(Y)$ on $\Pbar_Y$;
    \item The action of $G_Y$ on $\Pbar_Y$ preserves the set of $\Delta(Y)$-chambers and $\aut(Y)$ is equal to the stabiliser of $\Nef^{e}_Y$ in $G_Y$.
\end{enumerate}
\end{proposition}
\begin{proof}

Both assertions follow directly from the definitions.
Since $\Delta(Y)$ is preserved by $G_Y$, so is the set of $\Delta(Y)$-chambers. By standard arguments for Weyl groups, $W(Y)$ acts simply transitively on the set of $\Delta(Y)$-chambers and any chamber is a fundamental domain.
Since automorphisms preserve the nef cone, $\aut(Y) \cap W(Y) = 1$ and $\aut(Y)$ is the stabiliser of $\Nef^{e}(Y)$.
\end{proof}

Set
\[G_{X|Y}^0 = \{f \in O(S_X) \mid f(S_+) = S_+, f|S_X^\vee/S_X =\id\}.\]

\begin{theorem}
Let $Y$ be an Enriques surface over an algebraically closed field of characteristic $p\neq 2$.
Then $G_{X|Y}^0 \subseteq G_{X|Y}$ and the quotient $G_{X|Y}/G_{X|Y}^0$ is a finite group.
\end{theorem}
\begin{proof}
Let $g \in G_{X|Y}^0$. Since reflections in roots act trivially on the discriminant group, $W(X,\epsilon) \subseteq G_{X|Y}^0$ . By \Cref{remark:split_weyl} we may hence assume that $g$ preserves the ample cone of $X$ and commutes with $\epsilon^*$. We will show that $g \in \aut(X,\epsilon)$.

First suppose that $k=\CC$ and denote by $T_X$ the transcendental lattice of the covering K3 surface $X$. Since $g$ acts trivially on the discriminant group of $S_X$, the isometry $g'=g \oplus \id_{T_X}$ preserves the overlattice $S_X \oplus T_X \subseteq H^2(X,\ZZ)$. By construction $g'$ is an effective Hodge isometry. By the strong Torelli theorem (see e.g. \cite[VII (11.1)]{bhpv:compact-complex-surfaces}) we find $\gamma \in \Aut(X)$ with $\gamma^*=g'$. Since $g'$ commutes with $\epsilon^*$, $\gamma$ commutes with $\epsilon$.
For $k$ of characteristic zero we can reduce to the case of complex numbers by standard arguments.

Now suppose that $k$ is of characteristic $p>2$ and the K3 cover $X$ of $Y$ is of finite height.
Then by \cite[Theorem 2.1]{maulik} we find an $S_X$ preserving lift $X_K$ of $X$ to characteristic zero
and a group homomorphism $sp: \Aut(X_{\overline{K}}) \to \Aut(X)$ compatible with pullback and specialisation, where $\overline{K}$ is an algebraic closure.
By \cite[Corollary 2.4]{maulik} (or the preceeding) the specialisation map identifies the nef cones of $X$ and $X_{\overline{K}}$. By the fist part $g$ is induced by an element in the image of $sp$ that commutes with $\epsilon$.

Finally suppose that the K3 cover $X$ of $Y$ is supersingular.
Then Ogus' Crystalline Torelli theorem \cite[Theorem II' and Theorem II'']{ogus:83}\cite[Theorem 5.1.9 for p=3]{bragg-lieblich} implies that $g \in \aut(X)$.
Since it commutes with $\epsilon^*$, $g$ descends to an automorphism of $Y$ as desired.

$G_{X|Y}^0$ is of finite index in $O(S_X,S_+)$ and therefore of finite index in its subgroup $G_{X|Y}$.
\end{proof}

The $2$-congruence subgroup $O(S_Y,\P_Y)(2)$ is defined as the kernel of
$O(S_Y,\P_Y) \to O(S_Y \otimes \FF_2)$.
The following proposition is surely known to experts.
For $k=\CC$ it is \cite[Theorem 5.5.1]{cdl:enriquesI}.
\begin{proposition}\label{prop:G0}
Let $Y$ be an Enriques surface over an algebraically closed field $k$ of characteristic $p\neq2$.
Then $G_0:=O(S_Y,\P_Y)(2) \subseteq G_Y$ is a normal subgroup.
\end{proposition}
\begin{proof}
Let $f\in O(S_Y,\P_Y)(2)$. Note that $f$ acts trivially on the discriminant group of $S_Y(2) \cong S_+$
Hence $f'=f\oplus \id_{S_-}$ lies in $G^0_{X|Y} \subseteq G_{X|Y}$.
Since $G_Y$ is the image of $G_{X|Y}$ by \Cref{prop:GYG_+}, $f \in G_Y$.
\end{proof}
Denote by $\Gbar_\pm$ the image of $G_\pm$ under the natural map $O(S_\pm) \to O(S_\pm^\vee/S_\pm)$.
\begin{remark}
It is known in most cases that the quotient $G_{X|Y}/G_{X|Y}^0$ is a finite cyclic group.
Over the complex numbers this follows from \cite[Corollary 3.3]{nikulin:auto}, for Enriques surfaces whose K3 cover is supersingular this follows from \cite[Proposition 3.1]{jang:representation} and if the K3 cover is of finite height and $p\geq 23$, from \cite[Proposition 3.10]{jang:representation}.  \\
\end{remark}
\begin{remark}
Let $X$ be the K3 cover of a complex Enriques surface $Y$. If \(\pm 1\) are the only Hodge isometries of the transcendental lattice $T_X$ of of $X$, then $G_{X|Y}/G_{X|Y}^0$ is generated by the covering involution $\epsilon$. In this case, \cite[(3.4)]{brandhorst_shimada:tautaubar} we have $\Gbar_+ \cong \Gbar_-/\{\pm 1\}$.
\end{remark}

The effective nef cone encodes the rational curves and elliptic fibrations on $Y$ as follows.
\begin{proposition} \label{prop:class-smooth-rational-curve-facet}
Let $r \in S_Y$ with $r^2 = -2$. Then $r$ is the class of a smooth rational curve on $Y$ if and only if $r^\perp \cap \Nef^{e}_Y$ is a facet of the effective nef cone.
\end{proposition}
\begin{proof}
Follows from \cite[Proposition 2.2.1]{cdl:enriquesI} noting that the nef cone and its effective version have the same facets.
\end{proof}

\begin{proposition}\cite[VIII (16.1)]{bhpv:compact-complex-surfaces} \cite[2.2.9]{cdl:enriquesI}\label{prop:class-ell-fibr}
    Let $f \in S_Y$ be primitive and nef with $f^2 = 0$, then $2f$ is the numerical class of the fiber of an elliptic fibration $Y \to \mathbb{P}^1$.
\end{proposition}

Summing up, if we denote by $\R(Y)$ the set of numerical classes of smooth rational curves of $Y$, it holds $\R(Y) \subseteq \Delta^+(Y)\subseteq \Delta(Y)$
and $[C] \mapsto [C]^\perp \cap \Nef^{e}_Y$ is a bijection between $\R(Y)$ and the facets of $\Nef^{e}_Y$.
Similarly, the isotropic rays of the effective nef cone of $Y$ correspond to elliptic fibrations of $Y$.

\section{The quotient \texorpdfstring{$\Nef^e(Y)/\Aut(Y)$}{Nef(Y)/Aut(Y)}} \label{sect:cone-conj}
The effective nef cone and the automorphism group of an Enriques surface are notoriously difficult to control. For instance $\Aut(Y)$ and $\R(Y)$ can be empty, finite or infinite. However, the cone conjecture tells us that the action of $\Aut(Y)$ on the effective nef cone $\Nef^{e}_Y$ has a rational polyhedral fundamental domain.
We will see in this section how to control the quotient $\Nef^{e}_Y/\Aut(Y)$. \\

Before we can prove the main result, we need a preparatory Lemma.
For $x \in \Pbar_Y$, we set $\Delta(x) = \{\delta \in \Delta(Y) \mid x \in \delta^\perp\}$ and $W(x) = W(\Delta(x))$.
\begin{lemma}\label{lem:transitiveDxWx}
Let $x \in \Pbar_Y$. Then $W(x)$ acts simply transitively on the set of $\Delta(x)$-chambers.
\end{lemma}
\begin{proof}
This is a well known fact for Weyl groups. See e.g. \cite[Chapter 8 Proposition 2.6]{huybrechts:k3} for an ad-hoc proof in a similar setting.
\end{proof}

\begin{lemma}\label{lem:transitive}
Let $x \in \Pbar_Y$. Then the Weyl group $W(x)$ acts simply transitively on the set of $\Delta(Y)$-chambers whose closure contains $x$.
\end{lemma}
\begin{proof}
 Let $\mathcal{C}$ be the set of $\Delta(Y)$-chambers $C$ with $x$ contained in the closure $\overline{ C}$ of $C$ in $
 \Pbar_Y$.
 We denote by $\mathcal{D}$ the set of $\Delta(x)$-chambers.
 Since $\Delta(x) \subseteq \Delta(Y)$, any $\Delta(Y)$-chamber $C$ is contained in a unique $W(x)$-chamber $D$. Setting $\varphi(C)=D$ defines a surjection $\varphi\colon \mathcal{C} \to \mathcal{D}$. We claim that it is injective as well.

If $x \in \P_Y$ is an interior point, we let $B$ be an $\epsilon$-ball around $x$.
The collection of $\Delta(Y)$-hyperplanes is locally finite in $\P_Y$.
Therefore, choosing $\epsilon$ small enough, we can achieve the following:
If $r \in \Delta(Y)$ is such that $r^\perp \cap B\neq \emptyset$, then $x \in r^\perp$, i.e. $r \in \Delta(x)$.
If $x\in \Pbar_Y$ is a point in the boundary, let $B=\{y \in \P_Y \mid (x.y)^2/y^2 < \epsilon\}$ be (the cone over) an $\epsilon$-horoball at $x$.
By \cite[Corollary 3.12]{shimada:k3_auto} we can once again choose $\epsilon$ small enough to obtain the same property for $B$ as before.

Note that every $\Delta(Y)$-chamber $C$ with $x \in \overline{C}$ intersects $B$.
Let now $C, C' \in \mathcal{C}$ with $\varphi(C) = \varphi(C')=:D$. Choose $y \in B\cap C$ and $y' \in B \cap C'$. Then the line segment $l$ from $y$ to $y'$ is contained in $D\cap B$ since $B$ and $D$ are convex and $C,C' \subseteq D$.
If for $r \in \Delta(Y)$, its hyperplane $r^\perp$ separates $C$ and $C'$, then it must intersect $l$. In particular, $r^\perp \cap B \neq \emptyset$, hence $r \in \Delta(x)$ by our choice of $\epsilon$.
But the segment $l$ is contained in the chamber $D$ and therefore such $r$ cannot exist. We conclude $C=C'$ and the map $\varphi$ is indeed injective.

The Weyl group $W(x)$ acts equivariantly on both the domain and codomain of $\varphi \colon \C \to \D$. By \Cref{lem:transitiveDxWx}, $W(x)$ acts simply transitively on the set of $\Delta(x)$-chambers $\D$.
Therefore it also acts transitively on $\C$.
\end{proof}

\begin{theorem}\label{thm:NefmodAut}
The natural map $\Nef^{e}_Y / \aut(Y) \to \Pbar_Y/G_Y$ is bijective.
\end{theorem}
\begin{proof}
Let $x \in \Pbar_Y$. Since $\Nef^{e}_Y$ is a fundamental domain for the action of $W(Y)$ on $\Pbar_Y$ (by \Cref{prop:fundamental_domain}), we can find an element $\alpha$ of the Weyl group $W(Y)$ such that $x^\alpha\in \Nef^{e}_Y$. Therefore the map is surjective.
Let $x, y \in \Nef^{e}_Y$ such that there is $g \in G_Y$ with $x=y^g$.
Then $(\Nef^{e}_Y)^g$ is a $\Delta(Y)$-chamber containing $x$. Since $W(x)$ acts transitively on such chambers by \Cref{lem:transitive}, we find an $h \in W(x)$ with
$(\Nef^{e}_Y)^{gh} = \Nef^{e}_Y$. Then $gh\in\aut(Y)$ (\Cref{prop:fundamental_domain}) and it satisfies $y^{gh}=x^h=x$ as desired.
\end{proof}

Recall that $\HH^9 := \{x \in \P_Y \mid x^2=1\}$ is a $9$-dimensional hyperbolic space. It comes equipped with a natural volume.
\begin{corollary}\label{corollary:nefmodaut}
Let $Y$ be an Enriques surface.
The hyperbolic volume of the fundamental domain $\F_Y$ of $\aut_Y$ on $\Nef(Y)$ is given by
\[\mbox{vol}(\F_Y\cap \HH^9) = \left[O(S_Y \otimes \FF_2) : \Gbar_Y\right] 1_{\mathrm{Vin}}= 1_{\mathrm{BP}}/\#\Gbar_Y\]
where $1_{\mathrm{Vin}}$ is the hyperbolic volume of a Vinberg chamber and $1_{\mathrm{BP}}=\mbox{vol}(\F_{Y_0}\cap \HH^9)$ for a general Enriques surface $Y_0$.
\end{corollary}
\begin{proof}
This follows from the fact that $\Nef^e_Y/\aut(Y)\cong \Pbar_Y / G_Y$ (cf. \Cref{thm:NefmodAut}) is the quotient of $\Pbar_Y/G_0 \cong \Pbar_{Y_0}/G_{Y_0}$ by
$\Gbar_Y$.
\end{proof}

Let $\E(Y)$ be the set of elliptic fibrations $\phi\colon Y \to \mathbb{P}^1$. We identify it with the set of primitive isotropic rays of $\Nef^{e}_Y$ by identifying $\phi$ with the numerical class $f$ of a half fiber.

\begin{corollary}\label{cor:elliptic_bijection}
Let $\E(\Pbar_Y)$ denote the set of primitive isotropic classes of $\Pbar_Y$. Then the natural map
\[\E(Y) / \aut(Y) \to \E(\Pbar_Y)/G_Y\]
 is a bijection.
\end{corollary}
\begin{proof}
This follows from the fact that $\E(Y)$ respectively $\E(\Pbar_Y)$ is characterized as the subset of $\Nef^{e}_Y$, respectively $\Pbar_Y$, whose elements are represented by primitive isotropic classes.
\end{proof}

\section{Orbits of elliptic fibrations} \label{sect:orbits-fibrations}
In this section we count the $\Aut(Y)$-orbits of elliptic fibrations on an Enriques surface $Y$ with their ramification degree in terms of its Vinberg group $\Gbar_Y$. Then we explain how to recover the singular fibers of an elliptic fibration in our setting.

\begin{corollary}\label{prop:EYbar/GYbar}
Set $\E(S_Y \otimes \FF_2)$ as the set of non-zero isotropic classes of $S_Y\otimes \FF_2$.
Then
\[\E(Y)/\aut(Y) \cong \E(\Pbar_Y)/G_Y \cong \E(S_Y \otimes \FF_2)/\Gbar_Y.\]
In particular, two elliptic fibrations $\phi_1$,$\phi_2$ of $Y$ with half fibers $f_1$ and $f_2$ are isomorphic if and only if their classes
$\overline{f_i} \in S_Y\otimes \FF_2$ lie in the same $\Gbar_Y$ orbit.
\end{corollary}
\begin{proof}
We have a surjection between $\E(\Pbar_Y)/G_0$ and $\E(S_Y \otimes \FF_2)$, which is injective, because both sets have cardinality $527$. For the first one see \cite[p. 397, table]{barth-peters:very_general_enriques}, the second one is commonly known see e.g. \cite[(13.1)]{kneser}. The map
induces a bijection $\E(\Pbar_Y)/G_Y \cong \E(S_Y \otimes \FF_2)/\Gbar_Y$
because $\Gbar_Y= G_Y/G_0$.
\end{proof}

Let $Y_0$ be an Enriques surface such that $S_-=0$ and $G_{X_0|Y_0}=G_{X_0|Y_0}^0$, for instance a very general complex Enriques surface.
Then $G_0 \cong G_{X_0|Y_0}^0 = G_{X_0|Y_0} \cong G_{Y_0}=\aut(Y_0)$ and $\Nef^{e}_{Y_0}= \Pbar_{Y_0}$ (\cite[Section 0]{barth-peters:very_general_enriques} and \cite{martin:automorphisms_of_unnodal_enriques_surfaces}).
By the previous \Cref{thm:NefmodAut} we obtain a natural map
\[\pi\colon \Nef^{e}_{Y_0}/\aut(Y_0)\cong\Pbar_Y/G_0 \to \Pbar_Y/G_Y \cong \Nef^{e}_Y/\aut(Y)\]
with finite fibers and in the first isomorphism we identify $\Pbar_Y$ with $\Pbar_{Y_0}$ by any isometry $S_{Y_0}\cong S_Y$ mapping $\Pbar_{Y_0}$ to $\Pbar_Y$.
\begin{definition}\label{def:ramificaton}
For $f \in \E(Y)$ an elliptic fibration define its ramification degree $e_f$ as the cardinality of $\pi^{-1}(\{f\aut(Y)\})$.
\end{definition}
Over the complex numbers this can indeed be interpreted as a ramification degree, see \Cref{sect:moduli-general}.
The ramification degree $e_f$ of an elliptic fibration $f \in \E(Y)$ is computed as follows.
\begin{proposition}
 Let $\overline{f}$ be the class of $f \in \E(Y)$ in $S_Y \otimes \FF_2$. Denote by $\stab(\Gbar_Y,\overline{f})$ its stabiliser in $\Gbar_Y$. Then
 \[e_f = \#\left( \overline{f} \cdot \Gbar_Y \right)= [\Gbar_Y:\stab(\Gbar_Y,\overline{f})].\]
\end{proposition}

\begin{theorem}\label{thm:elliptic_fibrations}
The weighted count of isomorphism classes of elliptic fibrations $Y \to \mathbb{P}^1$ is $527$. More precisely,
\[527 = \sum_{f \in \E(Y)/\Aut(Y)} e_f.\]
\end{theorem}
\begin{proof}
 The order of $\E(\Pbar_Y)/G_0$ is $527$. The group $\Gbar_Y$ acts on it. Writing $\E(\Pbar_Y)/G_0$ as a disjoint union of
 its $\Gbar_Y$ orbits and applying \Cref{cor:elliptic_bijection} yields the claim.
\end{proof}
\begin{corollary}
 A complex Enriques surface has at most $527$ isomorphism classes of elliptic fibrations.
 This number is attained for a very general Enriques surface.
\end{corollary}

In what follows we show how to recover the ADE types and multiplicity of the reducible singular fibers
 of an elliptic fibration from the class $\overline{f} \in \E(S_Y\otimes \FF_2)$ of a half fiber.
\begin{lemma}\label{lem:injective}
Let $\Phi^+ \subseteq E_8$ be a positive root system and $\overline{\Delta}(E_8):= \{x \in E_8 \otimes \FF_2 \mid x^2/2 = 1\}$.
We consider the set of roots $\Phi$ of $E_8$ as a graph by joining $x,y \in \Phi$ with an edge if and only if $x.y \equiv 1 \mod 2$ and similarly for its image $\overline{\Phi}$ in $\overline{\Delta}(E_8)$.
Then the natural map $\Phi^+ \to \overline{\Delta}(E_8)$ is an isomorphism of graphs
\end{lemma}
\begin{proof}
We already know that for $\Phi^+$ a positive root system, it embeds into $E_{8}\otimes \FF_2$ \cite[Lemma 3.8]{brandhorst_shimada:tautaubar}. Hence the natural map is injective.
Since both sets are of order $120$, it is bijective. Moreover, the map is compatible with the bilinear forms defining the edges. Hence it is an isomorphism of graphs.
\end{proof}

\begin{remark}
 To recover the type of a root system from its graph, one first decomposes the graph into its connected components and is thus reduced to considering irreducible root systems only. Now an irreducible root system is uniquely determined by the number of its edges and vertices.
\end{remark}
Let $\overline{\Delta}(Y)$ denote the image of $\Delta(Y)$ in $S_Y\otimes \FF_2$.
For $f \in \E(Y)$ we denote by $\Delta^+(Y, f)$ the image of $\Delta^+(Y) \cap f^\perp$ in $f^\perp/\ZZ f$.
Recall that $\Delta^+(Y,f)$ is a positive root system consisting of effective split roots supported on a fiber of $f$. Its ADE-type is the ADE-type of the reducible singular fibers of $f$.
Denote by $\overline{\Delta}(Y,  \overline{f})$ the image of \(\overline{f}^\perp \cap \overline{\Delta}(Y) \) in  \(\overline{f}^\perp/ \FF_2 \overline{f}.\)
\begin{proposition}\label{prop:ADE-type}
 Let $f \in \E(Y)$.
 Then the  natural map \(f ^\perp / \ZZ f \to \overline{f}^\perp / \FF_2 \overline{f}\) induces an isomorphism
 \[\Delta^+(Y, f) \to \overline{\Delta}(Y,\overline{f})\]
 of graphs.
 In particular, the connected components of \(\overline{\Delta}(Y,\overline{f})\) correspond to the reducible singular fibers of the fibration induced by $f$, and their ADE-types correspond.
 \end{proposition}
\begin{proof}
By \Cref{lem:injective} the natural map is injective on the positive root system \(\Delta^+(Y,f)\). By the definitions its image is \(\overline{\Delta}(Y,\overline{f})\) because $\overline{\Delta}(Y)$ is the image of $\Delta(Y)$ under reduction mod $2$.
\end{proof}
Recall that an elliptic fibration $Y \to \mathbb{P}^1$ of an Enriques surface has exactly two fibers of multiplicity $2$.
We want to decide whether a reducible fiber is multiple or not in terms of $\overline{f}$ and $\overline{\Delta}(Y)$.
\begin{lemma}\label{lem:6R}
Let \(R\subseteq E_{10}\) be an irreducible root lattice and $R'= R\otimes \QQ \cap E_{10}$ its primitive closure.
Then \(6 R' \subseteq R\).
\end{lemma}
\begin{proof}
Since \(R'\) is even, \( R'/R \) is a totally isotropic subgroup of the discriminant group of $R$ (with respect to the discriminant quadratic form). By assumption \(R\) is isomorphic to \(A_n (n\leq 9)\); \(D_n (4 \leq n\leq 9)\); \(E_6, E_7, E_8\). A case by case check reveals that any totally isotropic subgroup $T$ of the discriminant group of $R$ satisfies $2T =0$ except for $A_8$ where $3 T =0$. Thus in any case \(6T =0 \).
\end{proof}
\begin{proposition}\label{prop:multiple}
Let $f \in \E(Y)$ and $C$ be a connected component of $\overline{\Delta}(Y, \overline{f})$.
Set
$\psi \colon \overline{f}^\perp \to \overline{f}^\perp / \FF_2 \overline{f}$ and
 $\overline{\Delta}(C) := \overline{\Delta}(Y) \cap \psi^{-1}(C)$.
Then the fiber corresponding to $C$ is multiple if and only if
$\overline{f}$ is contained in the $\FF_2$-linear hull of $\overline{\Delta}(C)$.
\end{proposition}
\begin{proof}
 Let $R \subseteq S_Y$ be the lattice generated by the irreducible components of the fiber corresponding to $C$. Then $[\ZZ f : \ZZ f \cap R] \in \{1, 2\}$ and the fiber is multiple if and only if
 the index is $1$. To determine the index, it suffices to check if $f \in R$.
 Let $R' = R \otimes \QQ \cap S_Y$ be the primitive closure of $R$ in $S_Y$.
 Then the linear hull of $\overline{\Delta}(C)$ is $(R+2R')/2R'$.
 Clearly, if $f \in R$, then $\overline{f} \in (R+2R')/2R'$. If conversely $\overline{f} \in (R+2R')/2R'$,
 then we can write \(f = r + 2r'\) with $r \in R$ and $r' \in R'$.
 Then \(3f = 3r + 6r' \in R+6R'\) which equals \(R\) by \Cref{lem:6R}. Thus \(f = 3f - 2f\) lies in \(R\).
\end{proof}

\section{Orbits of nef and big divisors}

Here we extend the results on elliptic fibrations to arbitary nef divisors.
To this end let $\overline{O(S_Y,h)}$ denote the image of $O(S_Y,h)$ in $O(S_Y\otimes \FF_2)$. The following theorem lets us count
$\Aut(Y)$-orbits of nef polarizations of $Y$ with their ramification degree. We will see in \Cref{thm:branch} that this terminology is justified over $\CC$.\\

For an element  $h\in S_Y$ we denote by $hO(S_Y)\subseteq S_Y$ its orbit by $O(S_Y)$. The \emph{ramification degree} of $h'=h^g \in hO(S_Y)$ is defined as the index
\[r_{h'} = \left[\Gbar_Y \colon \Gbar_Y \cap \overline{g}^{-1}\overline{O(S_Y,h)}\overline{g}\right].\]

\begin{theorem}\label{thm:fiber_charp}
Let $Y$ be an Enriques surface and $h \in S_Y$ with $h^2\geq 0$.
Then there is a bijection
\[\left(\Nef(Y)\cap hO(S_Y)\right) / \Aut(Y) \cong \overline{O(S_Y,h)}\backslash O(S_Y\otimes \FF_2)/\Gbar_Y,\]
and
\[\left[O(S_Y\otimes \FF_2):\overline{O(S_Y,h)}\right]= \sum_{h'}r_{h'},\]
where the sum runs over a complete set of representatives of $\left(\Nef(Y)\cap hO(S_Y)\right)/\Aut(Y).$
\end{theorem}

\begin{proof}
We use that $\Nef^e(Y)/\Aut(Y) \cong \Pbar_Y/G_Y$ and that
$\Pbar/G_0 \to \Pbar_Y/G_Y$ is a quotient by $\Gbar_Y$.
The statements follow since $\left(\Nef(Y)\cap hO(S_Y)\right)$ is a subset of $\Nef^e(Y)$, by applying the usual formulas for double cosets.
\end{proof}
\begin{remark}
For Enriques surfaces $Y$ over $\mathbb{C}$, the count of numerical $h$-quasi polarizations on $Y$ up to isomorphism is also discussed in \cite[(5.7.1)]{cdl:enriquesI}.
In our notation their formula is
$\left[O(S_Y \otimes \FF_2): \langle\overline{O(S_Y,h)}, \aut(S)\rangle\right]$.
This is in contradiction with our result and \cite[Table 8.3]{enriquesII}.
\end{remark}
\begin{remark}
 There is a unique orbit of primitive elements $e \in E_{10}$ with $e^2=0$. One can show that
 $\overline{O(E_{10},e)} \backslash O(E_{10}\otimes \FF_2) \cong \E(E_{10}\otimes \FF_2)$.
 This reconciles \Cref{thm:fiber_charp} and \Cref{thm:elliptic_fibrations}.
\end{remark}

For $h \in S_Y$ with $h^2>0$, set $\Delta(Y,h) = \{x \in \Delta(Y) \mid x.h=0, x^2=-2\}$ and let $\overline{\Delta(Y,h)}$ denote its image in $S_Y \otimes \FF_2$.
Note that the lattice $h^\perp$ is negative definite and therefore
\[\Delta(Y,h)=\{r \in h^\perp \mid r^2 = -2, \overline{r} \in \overline{\Delta}(Y)\}\]
can be computed by enumerating short vectors in $h^\perp$ and some linear algebra.
\begin{proposition}
    Let $Y$ be an Enriques surface and $h \in S_Y$ nef and big.
    Then some multiple of $h$ defines a morphism $Y \to \mathbb{P}^n$
    which contracts the set of $\left(-2\right)$-curves in $h^\perp$ to ADE singularities.
    Let $R \subseteq h^\perp$ be their $\ZZ$-span and $\Phi^+$ the set of effective roots of $R$. Then $\Phi^+ \to \overline{\Delta(Y,h)}, r\mapsto \overline{r}$ is an isomorphism of graphs.
    In particular, we can recover the ADE type of $R$.
\end{proposition}
\begin{proof}
By \cite[Lemma 3.8]{brandhorst_shimada:tautaubar} $\Phi^+ \to R/2R$ is injective. Let $\delta_1,\dots \delta_n$ be the fundamental root system of $\Phi^+$, which forms a basis of $R$. In particular, each element of $\Phi^+$ is uniquely determined by its intersections with the $\delta_i$ mod $2$. These are preserved under $
\Phi^+ \to S_Y\otimes \FF_2$. Hence this map is injective.

For surjectivity note that $\Delta(Y,h)$ is the set of splitting roots whose hyperplanes cut out the (minimal) face of the effective nef cone of $Y$ containing $h$. Then we have $\Phi^+=\Delta(Y, h)^+:=\Delta^+(Y)\cap h^\perp$
\end{proof}

\section{Applications to \texorpdfstring{\((\tau,\taubar)\)}{(tau,taubar)}-generic Enriques surfaces} \label{sect:tau-taubar}
We slightly generalize the definition of a $(\tau,\taubar)$-generic Enriques surface in \cite{brandhorst_shimada:tautaubar} and extend it to any characteristic not two.

Let $R \subseteq E_{10}$ be a negative definite root lattice.
Then the primitive closure $\overline{R} \subseteq E_{10}$ is a negative definite root lattice as well and the $O(E_{10})$ orbit of $R$ is uniquely determined by the root types
$(\tau(R),\tau(\overline{R}))$. There are precisely $184$ orbits of negative definite root sublattices of $E_{10}$ \cite[Proposition 1.2]{brandhorst_shimada:tautaubar}.

Let $R_2$ be a copy of $R=:R_1$. Define
$M_R$ to be the lattice in the quadratic space $V = (E_{10}\oplus R_2)(2) \otimes \QQ$ generated by $E_{10}$ and $\{(v,v)/2 \in V \mid v \in R\}$. It comes with a natural injection $\omega_R \colon E_{10} \to M_R$ which rescales the inner products by $2$.

\begin{definition}\label{def:tautaubar-generic}
An Enriques surface $Y$ over an algebraically closed field of characteristic not $2$ is said to be $(\tau,\taubar)$-generic if
there exists a root sublattice $R \subseteq E_{10}$ such that $(\tau,\taubar)= (\tau(R),\tau(\overline{R}))$ and the following hold:
\begin{enumerate}
\item there exist isometries $g \colon E_{10} \to  S_Y$
and $\tilde{g} \colon M_R \to S_X$ such that the following diagram commutes.
\begin{equation*} 
\begin{tikzcd}
   E_{10} \arrow[r]{}{\omega_R}\arrow[d]{}{g} &M_R \arrow[d]{}{\tilde{g}}\\
 S_Y\arrow[r]{}{\pi^*} &S_X
\end{tikzcd}
\end{equation*}
\item $G_{X|Y}$ is generated by $G_{X|Y}^0$ and the covering involution $\epsilon$.
\end{enumerate}
\end{definition}

\begin{remark}
In \cite{brandhorst_shimada:tautaubar}, the definition of a $(\tau,\taubar)$-generic complex Enriques surface is slightly more restrictive. There it is required that the group of integral Hodge isometries of the transcendental lattice $O(T_X, \omega_X)=\{\pm 1\}$. This implies our condition (2) in \Cref{def:tautaubar-generic}.
Over $\CC$, the $(\tau,\taubar)$-generic Enriques surfaces have $10-\rk R$ moduli.
\end{remark}
Let $Y$ be a $(\tau,\taubar)$-generic Enriques surface. To ease notation,
we identify $S_X$ and $M_R$ via $\tilde{g}$ and $S_Y$ and $E_{10}$ via $g$.
We have $R_1 \subseteq S_Y$ and $R_2(2)\subseteq S_-$.
Let $\pi_-\colon S_X \to S_-$ be the orthogonal projection. By \cite[Lemma 3.5 (1)]{brandhorst_shimada:tautaubar}, $\pi_-(2S_X)=R_2(2)$ and $S_-$ is the primitive closure of $R_2(2)$ in $S_X$.
\begin{lemma}\label{lemma:tautaubar_delta_and_GbarY}
Let $Y$ be a $(\tau,\taubar)$-generic Enriques surface and $R \subseteq S_Y$ a root lattice of type
$(\tau,\taubar)$ fitting in the commutative diagram from \Cref{def:tautaubar-generic}.
Let $\Phi(R)$ be the set of roots of $R$. Then
\begin{enumerate}
    \item $\overline{\Delta}(Y)$ is the image $\overline{\Phi(R)}$ of $\Phi(R)$ under $S_Y \to S_Y \otimes \FF_2$;
    \item $\Phi(R)/\{\pm 1\}$ injects into $S_Y \otimes \FF_2$;
    \item $\Gbar_Y$ is the image of $W(R)$ in $O(S_Y\otimes \FF_2).$
\end{enumerate}
\end{lemma}
\begin{proof}\hfill
\begin{enumerate}
    \item By the definitions the elements $v \in \Phi(R)$ are splitting roots, indeed $(v,v)/2 \in S_X\cong M_R$.
    Conversely if $s \in S_Y$ is a splitting root, then we find $v \in 2\pi_-(S_X)=R_2(2)$ with $v^2=-4$ and $(\pi^*(s),v)/2 \in S_X$.
    Thus $((\pi^*(s)-v)/2,0)=(\pi^*(s),v)/2 - (v,v)/2 \in S_X \cap S_+^\vee=S_+=\pi^*(S_Y)$. This shows that $\overline{s}=\overline{v} \in \overline{\Phi(R)}$.
    \item This follows from \cite[Lemma 3.8]{brandhorst_shimada:tautaubar} since $\overline{R}$ is a negative definite root lattice and $\overline{R} \otimes \FF_2 \to S_Y \otimes \FF_2$ is injective.

    \item To warm up note the following: since by (1) the roots of $R$ all split, $W(R) \subseteq G_Y$.
    The other inclusion is harder.
    Let $\Gbar_{X|Y}$ denote the image of $G_{X|Y}^0$ in the orthogonal group of the discriminant group of $S_+ \oplus S_-$, and denote correspondingly $G_{\pm}^0$, respectively $\Gbar_{\pm}^0$, as the image of $G_{X|Y}^0$ in $O(S_\pm)$, respectively $O(S_\pm^\vee/S_\pm)$.
    By the genericity assumption (2) in \Cref{def:tautaubar-generic} and since the covering involution $\epsilon$ acts trivially on $S_+$, $G_+ = G_+^0$.
    Recall the commutative diagram found in \cite[equation (3.4)]{brandhorst_shimada:tautaubar} (adapted to our slightly different notation)
    \[
    \begin{tikzcd}
        G_Y \cong G_{+}^0 \arrow[d] &G_{X|Y}^0 \arrow[l]\arrow[r]\arrow[d] & G_{-}^0 \arrow[d]\\
        \Gbar_Y \cong \Gbar_{+}^0  &\Gbar_{X|Y}^0\arrow[l,"\sim"'] \arrow[r]{}{\sim} & \Gbar_{-}^0
    \end{tikzcd}.
    \]
    Set \(\Phi=\{((r,r)/2 \in S_X \mid r \in \Phi_+\}\) where
    $\Phi_+$ is the set of roots of $R_1\subseteq E_{10}$ and $\Phi_-$ the ones of $R_2$.
    Then the above diagram restricts to
    \[
    \begin{tikzcd}
        W(\Phi_+) \arrow[d] &W(\Phi) \arrow[l]\arrow[r]\arrow[d] & W(\Phi_-) \arrow[d]\\
        \Gbar_Y\cong\Gbar_{+}^0  &\Gbar_{X|Y}^0\arrow[l,"\sim"'] \arrow[r]{}{\sim} & \Gbar_{-}^0
    \end{tikzcd}.
    \]
    By the proof of \cite[Theorem 3.11]{brandhorst_shimada:tautaubar}, $G_{-}^0 = W(R_2)=W(\Phi_-)$. Now the diagram yields the statement.
\end{enumerate}
\end{proof}

\begin{example}
Let $Y$ be an $(A_1,A_1)$-generic Enriques surface.
Then $S_- \cong A_1(2)$, $\overline{\Delta}(Y)$ consists of a single point, say $\overline{r}$. Further $\Gbar_Y \cong \Gbar_- \cong W(A_1)$ is of order $2$ and generated by the reflection in $\overline{r}$.
Let $\overline{f} \in \E(S_Y\otimes \FF_2)$.
If $\overline{f} \in \overline{r}^\perp$, then \(\overline{f}\) is fixed by $\Gbar_Y$ and there are \(2^{8}-1=255\) such vectors. The corresponding elliptic fibrations have a unique reducible singular fiber $C$. It is of type $\tilde{A_1}$ and multiplicity $1$ by \Cref{prop:multiple} since $\overline{f} \notin \langle \overline{\Delta}(C)\rangle = \FF_2 \overline{r}$.

The remaining $527-255=2 \cdot 136$ isotropic classes come in pairs of two, which is their weight. The corresponding fibrations have no reducible singular fibers.
\end{example}

\begin{example}
Let $Y$ be a \((2A_1,2A_1)\)-generic Enriques surface.
Write $E_{10} \otimes \FF_2= \bigoplus_{i=1}^5 u_i$ with $u_i$ generated by $f_i, e_i$ and $q(f_i)=q(e_i) =0$, \( f_i.e_i =1\).
Set $r_i = f_i+e_i$. Note that $q(r_i)=1$. Up to a change of basis, we may assume that $\overline{\Delta}(Y)=\{r_1,r_2\}$ and $\Gbar_Y$ is of order \(4\) and generated by the reflections in $r_1$ and $r_2$.
Write an isotropic vector as $\overline{f} = x+y$ with $x \in u_1 \oplus u_2$ and $y \in u_3 \oplus u_4 \oplus u_5$.
Note that there are $28$ elements in $u_3 \oplus u_4 \oplus u_5$ of square $1$, $35$ nonzero elements of square $0$, and the zero element, giving a total of $28+35+1=64=2^6$ as expected. We separate cases depending on $x$.
\begin{enumerate}
\item  \(x \in \{0, r_1, r_2, r_1+r_2\}\) and $y \neq 0$. There are \(126 = 35 + 28+28 + 35\) choices for $\overline{f}$ and $\Gbar_Y$ fixes $\overline{f}$. Thus $\overline{f}$ is of weight $1$.
We have $\overline{\Delta}(Y) \cap\overline{f}^\perp = \{r_1,r_2\}$. Further $r_1 +\overline{f} \neq r_2$ because $y\neq 0$.
Thus $\overline{\Delta}(Y,\overline{f})$ consists of two isolated points and the fibration corresponding to $\overline{f}$ has two fibers of type $A_1$.
Since $\overline{f}$ is not in the linear hull of $\overline{\Delta}(Y)$, the two fibers reduced.
\item \(x \in \{e_1+e_2 + \alpha r_1 + \beta r_2 \mid \alpha,\beta \in \{0,1\}\}\). There are \(4 \cdot 36\) choices for $\overline{f}$, $\#(\overline{f}\cdot\Gbar_Y)= 4$, $\overline{\Delta}(Y) \cap\overline{f}^\perp = \emptyset$.
In particular, the fibrations have weight $4$ and no reducible singular fibers.
\item  \(x \in \{e_i + \alpha r_1 + \beta r_2 \mid i \in \{1,2\}, \alpha,\beta \in \{0,1\}\}\).
      There are \(4\cdot 64\) choices for $\overline{f}$ and $\#(\overline{f}\cdot\Gbar_Y)=2$ in each case.
      The set $\overline{\Delta}(Y) \cap \overline{f}^\perp= \{r_i\}$ consists of a single point, so does its image $\overline{\Delta}(Y,\overline{f})$. Thus there is a single reducible fiber $C$ and it is of type $A_1$. It must be of multiplicity one because $\overline{f}$ is not contained in the linear hull of $\overline{\Psi}_C= \{r_i\}$.
\item \(\overline{f} = r_1 + r_2\). Then $\#(\overline{f}\cdot\Gbar_Y)=1$ and $\overline{\Delta}(Y) \cap \overline{f}^\perp =\{r_1,r_2\}$.
  Since $r_1 = \overline{f}+ r_2$, $\overline{\Delta}(Y,\overline{f}) = \{r_1 + \FF_2 \overline{f}\} \subseteq \overline{f}^\perp / \FF_2 \overline{f}$ consists of a single point.
  Thus there is a single reducible fiber $C$. It is of type $A_1$. Moreover it is multiple because
  $\overline{f}$ is in the linear hull of $\overline{\Delta}(C) = \{r_1, r_2\}$.
\end{enumerate}
As a sanity check we calculate that \(126 + 4 \cdot 36 + 4 \cdot 64 + 1 = 527.\)
\end{example}

For the next example we recall some notation from \cite[Section 3]{brandhorst_shimada:tautaubar}.
Let $D_\pm \subseteq S^\vee_\pm/S_\pm$
be the image of $S_X/(S_+
\oplus S_-)$ under the orthogonal projection.
The latter is the graph of the gluing anti-isometry $D_+ \cong D_-$. Set $D_-^\perp$ as the orthogonal of $D_-$ in $S_-^\vee/S_-$.
By \cite[Theorem 3.4]{brandhorst_shimada:tautaubar} we have
\begin{equation}\label{eqn:Dminus-perp}
    G_-^0 = \ker \left(O(S_-,D_-) \to O(D_-^\perp)\right).
\end{equation}
\begin{example}
Let $Y$ be a general Enriques surface such that $S_- \cong [-4m]$ for $m \geq 1$.
The case $m=1$ corresponds to $1$-nodal i.e. $(A_1,A_1)$-generic Enriques surfaces.
Let $m \geq 1$. Then $S_-$ does not contain any vectors of square $4$ and therefore  \(\Delta(Y)\) is empty as well and the effective nef cone coincides with $\Pbar_Y$. We claim that $\Gbar_Y = 1$.
To see this let $e$ be a generator of $S_-$ with $e^2 = -4m$. Then $D_- = 1/2 e$ and $D_-^\perp = 1/(2m) e$.
Now \cref{eqn:Dminus-perp} tells us that
elements of \(G_-\) act trivially on \(D_-^\perp\).
The only non trivial isometry of $S_-$ is $-1$.
But it does not act trivially on $D_-^\perp$ for $m>1$.
The claim $\Gbar_Y =1$ is proven.
Therefore $G_Y =G_0$ and $\Aut(Y) = G_0$ is the $2$-congruence subgroup of $O(S_Y,\P_Y)$.
There are precisely $527$-elliptic fibrations on $Y$ up to automorphisms.
\end{example}

\section{Moduli spaces of complex Enriques surfaces} \label{sect:moduli-general}
In this section we review the construction of moduli spaces of numerically quasi-polarized complex Enriques surfaces \cite{gritsenko-hulek} and extend it to cover the case of nef but not big divisors as well. This allows us to define moduli spaces of elliptic Enriques surfaces. As remarked in \cite[Remark 2.2]{gritsenko-hulek} it is not clear whether the spaces we construct represent any moduli-functor. But at least their points are in bijection with isomorphism classes of numerically quasi-polarized complex Enriques surfaces (see \cref{thm:moduli}).
\subsection*{Polarizations.}
A numerically quasi-polarized Enriques surface is a pair $(Y,h)$ consisting of an Enriques surface and a nef class $h \in S_Y$.
If $h$ happens to be ample, this is called a numerical polarization.
Note that this includes the case $h^2=0$ and $h=0$, which are usually excluded in the literature.

The notion of isomorphism of numerically quasi-polarized Enriques surfaces is the obvious one:
\begin{definition}
    Two numerically quasi-polarized Enriques surfaces $(Y_1,h_1)$ and $(Y_2,h_2)$  are said to be isomorphic if there is an isomorphism $\psi\colon Y_1 \to Y_2$ with $\psi^*(h_2) = h_1$.
\end{definition}
Let $Y$ be an Enriques surface with covering $K3$ surface $\pi\colon X\to Y$.
A \emph{marking} of $Y$ is an isometry $\varphi\colon S_Y\xrightarrow{\sim} E_{10}$ of the numerical lattice of the Enriques surface to the Vinberg lattice $E_{10}$.

\begin{definition}
Let $h \in E_{10}$ with $h^2\geq 0$. We call $(Y,h')$ a \emph{numerically $h$-quasi polarized} Enriques surface if for any (and hence every) marking $\varphi$ of $Y$ it holds that $\varphi(h') \in hO(E_{10})$.
\end{definition}

A marking $\varphi$ is the same as an isometry $\varphi\colon S_Y\left(2\right)\xrightarrow{\sim} E_{10}\left(2\right)$. The latter lattice has an embedding into the $K3$ lattice $\Lambda:=U^{\oplus 3}\oplus E_8^{\oplus 2}$, which is unique up to the action of \(O(\Lambda)\) \cite[Theorem 1.14.4]{nikulin}. We fix once and for all one such embedding and denote by $M\subseteq \Lambda$ the image and by $N=M^{\perp}\subseteq\Lambda$.
Recall that the pull-back by $\pi$ induces an isometry $S_Y(2)\cong\pi^*\left(S_Y\right)=:S_+$ with a primitive sublattice $S_+\subseteq S_X\subseteq H^2\left(X,\ZZ\right)$.
The marking $\varphi$ can be (non-uniquely) extended to a marking $\tilde{\varphi}\colon H^2\left(X,\ZZ\right)\stackrel{\cong}{\to}\Lambda$ of the covering $K3$ surface, which maps $S_- = S_+^{\perp}\subseteq H^2\left(X,\ZZ\right)$ isometrically into $N$ and $S_+$ onto $M$.

\subsection*{Periods.}
The Enriques involution acts as $-\id$ on the space of global holomorphic $2$-forms $\omega_X=H^0\left(K_X\right)\subseteq H^2\left(X,\CC\right)$, hence it is mapped to $N_{\CC}$ by the ($\CC$-linear extension of the) marking $\widetilde{\varphi}$. Moreover, the Riemann bilinear relations imply that the line $\widetilde{\varphi}_{\CC}\left(\omega_X\right)$ actually lies in
$$\Omega_N:=\left\{\left[x\right]\in\PP\left(N_{\CC}\right)\,\mid\,x^2=0, x.\overline{x}>0\right\}.$$
We call this point the \emph{period} of the unpolarized Enriques surface $Y$ with respect to the chosen $\widetilde{\varphi}$.

However not every point of $\Omega_N$ is the period of some Enriques surface with respect to some marking. Such points are precisely the complement of a certain union of hyperplanes $\mathcal{H}\subseteq\PP\left(N_{\CC}\right)$:
$$\mathcal{H}:=\bigcup_{r\in N,r^2=-2}\PP\left(r^{\perp}\right)=\bigcup_{r\in N,r^2=-2}\left\{\left[x\right]\in\PP\left(N_{\CC}\right)\,\mid\, r.x=0\right\},$$
which is obviously stable by the action of $O\left(N\right)$. Denote by
$$\Omega_N^\circ=\Omega_N\setminus\mathcal{H}$$
the set of actual periods.

\begin{definition}
    The \emph{period} of the numerically $h$-quasi polarized Enriques surface $(Y,h')$ with respect to $\varphi$ and its extension $\widetilde{\varphi}$ is
    \begin{equation} \label{eq:period}
    \left(\varphi\left(H\right),\widetilde{\varphi}_{\CC}\left(H^{2,0}(X)\right)\right)\in hO(E_{10})\times\Omega_N^\circ.
    \end{equation}
\end{definition}

\subsection*{The moduli spaces.}
Now we need to find a suitable group $G$ acting on $hO(E_{10})\times\Omega_N^\circ$ by identifying two periods if and only if the quasi-polarized Enriques surfaces are isomorphic.
It turns out to be the same as eliminating the dependence on the markings $\varphi$ and $\widetilde{\varphi}$.

\begin{definition}
    Let $G=O\left(\Lambda,M\right)$ be the subgroup of $O\left(\Lambda\right)$ preserving the chosen sublattice $E_{10}\left(2\right)\cong M\subseteq\Lambda$ (or equivalently its orthogonal complement $N=M^{\perp}$).
\end{definition}

Since $M\oplus N\subseteq\Lambda$ is a sublattice of finite index, any isometry in $G$ is uniquely determined by its restrictions $g$ resp. $h$ to $M$ resp. $N$, which satisfy a compatibility condition (define the same action in the discriminant groups $M^\vee/M \cong N^\vee/ N$). In particular we can identify $G$ with a subgroup of $O\left(E_{10}\right)\times O\left(N\right)$. Thus we also have an action of $G$ on our enriched period domain $hO(E_{10})\times\Omega_N^\circ$.
\begin{theorem}\label{thm:moduli}
    Let $h \in E_{10}$ with $h^2\geq 0$.
    There is a natural bijection between the set of isomorphism classes of numerically $h$-quasi-polarized Enriques surfaces 
    and the quotient
    \[\M_{En,h}:=  \left(hO(E_{10}) \times \Omega_N^\circ \right)/O\left(\Lambda,M\right)\]
    induced by the period map.
    We call $\M_{En,h}$ the moduli space of numerically $h$-quasi-polarized Enriques surfaces.
\end{theorem}

\begin{proof}
    Let $(Y,h')$ be a numerically $h$-quasi polarized Enriques surface. We first show that the $G$-orbit of its period is independent of the choices of a marking $\varphi$ and its extension $\widetilde{\varphi}$. Indeed, any two extensions of $\varphi$ differ by an element of $\{\id_M\} \times O^\sharp(N) \subseteq O(\Lambda, M)=G$ and any two choices of markings differ by an element of $O(E_{10})$. Since $G \to O(M)=O(E_{10})$ is surjective, this leads to the same $G$-orbit.

    Now we show that the map from the set of isomorphism classes of numerically $h$-quasi polarized Enriques surfaces to $\M_{En,h}$ which maps $(Y,h')$ to the $G$-orbit of its period is injective.
    Suppose that $(Y_1,h_1')$ and $\left(Y_2,h_2'\right)$ have the same period, that is, we can find markings $\varphi_i$ of $Y_i$, extensions $\widetilde{\varphi_i}$ to markings of $X_i$
    for $i=1,2$ and $\widetilde{g}=(g_1,g_2) \in G\subset O(M)\times O(N)$ with $\widetilde{\varphi}_2(h_2)^{\widetilde{g}} = \widetilde{\varphi}_1(h_1)$ and $ \widetilde{\varphi}_2(\omega_2)^{\widetilde{g}} = \widetilde{\varphi}_1(\omega_1)$ where $h_i' = \pi_i^*h_i$ and $\omega_i=H^{2,0}(X_i)$.
   Let
    $$k:=\widetilde{\varphi}_1^{-1} \circ \widetilde{g}\circ \widetilde{\varphi}_2\colon H^2\left(X_2,\ZZ\right)\stackrel{\cong}{\to} H^2\left(X_1,\ZZ\right).$$
    After replacing $\tilde g$ by \((-g,h)\) if necessary, we can assume that $k$ maps the positive cone of $X_2$ to the positive cone of $X_1$. By construction $k$ maps $\omega_2$ to $\omega_1$ and therefore
    $S_{X_2}$ to $S_{X_1}$. Since $\widetilde{g}$ preserves $M$, $k$ maps $(S_{X_2})_\pm$ to $(S_{X_1})_\pm$. In view of the purely lattice theoretic description of the set of splitting roots in \Cref{lem:splitting-root}, this implies $k(\Delta(Y_2))=\Delta(Y_1)$. Therefore
    the image of the ample cone of $Y_2$ is a $\Delta(Y_1)$-chamber of the positive cone of $Y_1$.
    By \Cref{lem:transitive} we find an
    element $\alpha$ of the Weyl group
    $W(h_2) \subseteq W(Y_2)$ such that
    $\varphi_1^{-1} \circ g \circ \varphi_2 \circ \alpha$ maps the ample cone of $Y_2$ to the ample cone of $Y_1$. By \Cref{prop:GYG_+} the map $G_{X_2|Y_2} \to G_{Y_2} \supseteq W(Y_2)$ is surjective. Therefore we can lift $\alpha$ to a Hodge isometry $\tilde{\alpha}$ of $H^2(X_2,\ZZ)$ preserving $(S_{X_2})_{\pm}$.
    By construction $k\circ \tilde{\alpha}$ is an effective Hodge isometry. By the strong Torelli theorem there is an isomorphism $\widetilde{\psi}\colon X_1 \to X_2$ with $\widetilde{\psi}^* = k \circ \tilde{\alpha}$. Since $\widetilde{\psi}^*$ maps $\left(S_{X_2}\right)_+$ to $\left(S_{X_1}\right)_+$, $\psi$ is compatible with the covering involutions, hence it descends to an isomorphism $\psi \colon Y_1 \to Y_2$ of the Enriques surfaces. By construction
    $$\psi^*(h_2)=k\left(\tilde{\alpha}\left(h_2\right)\right)=k\left(h_2\right)=\left(\widetilde{\varphi}_1^{-1} \circ \widetilde{g}\circ \widetilde{\varphi}_2\right)\left(h_2\right)=h_1.$$
     This shows $\left(Y_1,h_1\right)\sim(Y_2,h_2)$.

    For the surjectivity: Let $\left(h',v\right)\in O(M)h\times\Omega_N^\circ$. Then there is a marked Enriques surface $\left(Y,\varphi\right)$ with
    $$v=\left(\widetilde{\varphi}\right)_{\CC}\left(H^{2,0}(X)\right),$$
    and (up to replacing $\varphi$ by $-\varphi$),
    $\varphi^{-1}\left(h'\right)\in S_Y$ is a class in the positive cone, but not necessarily nef.
    It is however contained in the closure of some $\Delta(Y)$-chamber $C$. By \Cref{prop:fundamental_domain}, we find an element $\alpha$ of the Weyl group $W(Y)$ with $\alpha(C)$ the ample cone of $Y$. Set $\varphi' := \varphi \circ \alpha^{-1}$. Then $h''=\varphi'^{-1}(h')=\alpha\left(\varphi^{-1}\left(h'\right)\right)$ is nef, so $(Y,h'')$ is a numerically $h$-quasi-polarized Enriques surface with period $(h',v)$ via the markings $\varphi'$ and $\widetilde{\varphi}'=\widetilde{\varphi}\circ\widetilde{\alpha}$ for a suitable extension of $\alpha$ to $H^2(X,\ZZ)$.

\end{proof}

\begin{remark}
Our construction can be reconciled with the conventional construction of the moduli spaces $\M_{En,h}$ in \cite{gritsenko-hulek} as follows.
First of all one fixes $h \in E_{10}$ and a connected component $\D_N$ of $\Omega_N^\circ$.
Denote by $\Gamma_h$ the image of $O(\Lambda,M,h) \to O(N)$.
For each orbit $(h',\omega)G$ choose $g \in G$ with $(h')^g=h$ and $\omega^g \in \D_N$. It is well defined up to the action of the joint stabiliser $O(\Lambda, M, h, \D_N)$ of $h$ and $\D_N$ in $G$. Let $\Gamma_h^+$ be its image in $O(N)$. This gives isomorphisms
\begin{equation}
    \M_{En,h}:=(O(E_{10})h \times \Omega_N^\circ)/G \cong \Omega_N^\circ/\Gamma_h \cong  \D_N/\Gamma_h^+.
\end{equation}
Since $O^\sharp(N,\D_N) \subseteq \Gamma_h$, the latter is an arithmetic subgroup of $O(N,\D_N)$ and therefore the quotient
$\D_N/ \Gamma_h$ a $10$-dimensional quasi projective normal variety \cite{baily-borel} -- even for $h^2=0$ or $h=0$.  
\end{remark}
\section{Forgetful maps between moduli spaces of Enriques surfaces}\label{sect:moduli-fibrations}
Let $v \in E_{10}$ be a Weyl vector as defined in \Cref{sec:vinberg}. Recall that its stabiliser in $O(E_{10})$ is trivial. Therefore $O(\Lambda,M,v)= \{\id_M\}\times O^\sharp(N)$ and thus $\Gamma_v = O^\sharp(N)$.
Fix $h \in E_{10}$ with $h^2\geq 0$. Considering the groups
\[O^\#(N)=\Gamma_w \subseteq \Gamma_h \subseteq \Gamma_0=O(N),\]
we see that there are
finite ramified coverings of normal quasi-projective varieties:
\begin{equation}
\M_{En,v} \xrightarrow{\P^v_h}\M_{En,h} \xrightarrow{\P^h_0}{\to} \M_{En,0}
\end{equation}
The definition of $\Gamma_h$ depends on the choice of $h$ in its $O(E_{10})$-orbit. Different choices give conjugate groups and thus different maps $\P^v_h$ and $\P^h_0$. They fit in the obvious commutative diagram. See \cref{rmk:modular_w_h} for the modular interpretation of $\P^w_h$.

The map $\P^h_0$ has the modular interpretation of sending a (family of) numerically $h$-quasi-polarized Enriques surfaces $(Y,h')$
to the unpolarized (family) $(Y,0)$.  With the identification $\M_{En,0}=(\{0\} \times \Omega_N^\circ)/G \cong \Omega_N^\circ/O(N)=: \M_{En}$ we have
\[\P^h_0\colon \M_{En,h} \to \M_{En}, \quad (h'',\omega)G \mapsto O(N)\omega.\]

We are interested in the degree, the ramification locus and the ramification degrees of the map $\P^h_0$.

A ramified cover is a finite, surjective, proper holomorphic map $\pi \colon X
\to Y$ between connected, \emph{normal} complex spaces.
By the discussion in \cite[Chapter I Section 16]{bhpv:compact-complex-surfaces}
any point in $x \in X$ has a local degree $e_x$ also called ramification degree of $\pi$ at $x$.

 \begin{lemma}\label{lem:quotients}
 Let $X$ be a normal complex space and $H \leq G$ finite groups acting faithfully on $X$ such that the canonical projections
\[X \xrightarrow{\alpha} X/H \xrightarrow{\beta}  X/G\]
are morphisms of quasi-projective varieties.
Let $x \in X$.
\begin{enumerate}
    \item The degrees are $\deg \alpha = \#H$, $\deg \beta= [G:H]$ and $\deg (\beta\circ\alpha) = \#G$
    \item The (reduced) fiber over $Gx$ is in bijection with the double coset
\[\beta^{-1}(Gx) \cong  x \stab(G,x) \backslash G /H . \]
\item The ramification degrees of $\alpha$, $\beta\circ\alpha$ at $x$ and $\beta$ at $\alpha(x)$
are given by
\[e_{x}(\alpha) = \#\stab(H,x), \quad e_x(\beta\circ\alpha) = \# \stab(G,x)\]
\[e_{\alpha(x)}(\beta) = \left[\stab(G,x): \stab(H,x)\right].\]
 \end{enumerate}
 \end{lemma}
 \begin{proof}
 1) We have $\# G = \deg (\beta \circ \alpha)  = \deg \alpha \cdot \deg \beta = \# H \deg \beta$.

 2) The double coset gives the decomposition of $G x$ into $H$-orbits.

 3) 
Since $H$ acts transitively on the fiber of $\alpha$ over $\alpha(x)$, each point of the fiber has the same ramification degree $m=e_x$. Now
\[\deg \alpha = \# H = m \cdot \# H x = m [H : \stab(H,x)].\]
Thus $e_x=m=\# \stab(H,x)$. The same argument works for $\beta\circ \alpha$.

Let $X'$ be a sufficiently small open neighborhood of $x$ such that $e_x(\alpha\circ\beta) = \deg (\beta\circ\alpha|_{X'})$, $G'=\stab(G,x)$ and $H'=\stab(H,x)$. Let $\alpha',\beta'$ be the corresponding restrictions of $\alpha$ and $\beta$.
Note that the fiber of $\beta'\circ\alpha'$ over $\beta'\circ\alpha'(x)$ consists of a single point. Therefore
$e_{x}\alpha=\deg \alpha'$, $e_x(\beta\circ\alpha)=\deg( \beta'\circ\alpha')$ and $e_{\alpha(x)}(\beta)=\deg \beta'$.
Now the claim follows with (1) applied to $\alpha'$, $\beta'$ and $\beta'\circ\alpha'$.
\end{proof}

Recall that $O(E_{10},h)$ denotes the stabilizer of $h \in E_{10}$ in $O(E_{10})$.
Set \[\overline{O(E_{10},h)} := \im ( O(E_{10},h) \to O(E_{10}\otimes \FF_2)).\]
\begin{theorem}\label{thm:branch}
Let $h \in E_{10}$ with $h^2\geq 0$. Consider the finite covering
\[\P^h_0\colon \M_{En,h} \to \M_{En}.\]

The following hold.
\begin{enumerate}[leftmargin=*]
\item The degree of $\P^h_0$ is given by
\[\deg \P^h_0 = \left[O(E_{10}\otimes \FF_2) : \overline{O(E_{10},h)}\right].\]
\item The fiber of $\P^h_0$ over the (unpolarized) Enriques surface $Y$ is
\[\left(\P_0^h\right)^{-1}([Y,0]) \cong \overline{O(S_Y,h)} \backslash O(S_Y\otimes \FF_2) / \Gbar_Y  .\]
In particular, the number of numerical $h$-quasi polarizations on $Y$ up to isomorphism, i.e. up to $\aut(Y)$ is
$\# O(E_{10}\otimes \FF_2) / \# \langle  \overline{O(E_{10},h)}, \Gbar_Y\rangle$.
\item The ramification degree of $\P^h_0$ at a point $[h',\omega]$ in the fiber over $[Y]$ is given by
\[\left[\;\Gbar_Y: \Gbar_Y \cap \overline{O(S_Y, h')}\right].\]
\end{enumerate}
\end{theorem}
\begin{proof}
Recall that for a fixed $h$ we have a commutative diagram
\begin{equation}\label{eq:ramification}\begin{tikzcd}
    \M_{En,h} =(hO(E_{10}) \times \Omega_N^\circ)/G \arrow[r]{}{\sim}\arrow[d]{}{\P^h_0} & \Omega_N^\circ/\Gamma_h \arrow[ld]{}{\pi}\\
    \M_{En} = \Omega_N^\circ/O(N)&
\end{tikzcd}
\end{equation}
and the horizontal isomorphism is given by $(h',\omega)G \mapsto \omega' \Gamma_h $ where $\omega'$ is defined by  $(h,\omega') \in (h',\omega)G$.

1) We work on the side of $\pi$ and apply \Cref{lem:quotients} in the situation that
 $X = \M_{En,w}\cong \Omega_N^\circ/\O^\sharp(N)$, $H = \overline{\Gamma}_h \cong \overline{O(E_{10},h)}$, $G=\overline{\Gamma}_0 \cong O(E_{10} \otimes \FF_2)$, $\alpha = \P^w_h$, $\beta=\P^h_0$.

2) We work directly with $\P^h_0$.
Let $\omega$ be any period of $Y$ and $G_\omega$ be the image of $O(\Lambda,M,\omega) \to O(E_{10})$. The fiber is thus given by
$(h,\omega)O(\Lambda,M, \omega) = (hO(E_{10})\times \{\omega\})/G_\omega $. A marking gives identifications $E_{10}\cong S_Y$ and
$G_\omega \cong G_Y$.
By \Cref{prop:G0} we know $O^\sharp(S_Y(2)) \subseteq G_Y$ and hence
\begin{eqnarray*}
\left(\P_0^h\right)^{-1}([Y]) &\cong&  hO(S_Y)/ G_Y \\
&\cong&  O(S_Y,h) \backslash O(S_Y) / G_Y \\
&\cong&  \overline{O(S_Y,h)} \backslash O(S_Y\otimes \FF_2) / \Gbar_Y
\end{eqnarray*}

3) We work on the side of $\pi$ again and invoke \Cref{lem:quotients}.
Translating to our situation let $\omega$ be the period of $Y$ (with respect to some markings) and
$\omega g_2 \Gamma_h$ a point in the fiber of $\pi$ over $\Gamma_0 \omega$ where $g_2 \in \Gamma_0$.
Denote $[\omega]:= \omega\O^\sharp(N)$.
By the lemma the ramification degree is
\[[\stab(\overline{\Gamma}_0,[\omega]): \stab(\overline{\Gamma}_0,[\omega]) \cap \overline{g}_2^{-1} \overline{\Gamma}_h \overline{g}_2] \]
Using the isomorphism $E_{10}\otimes \FF_2 \cong M^\vee/M \cong N^\vee /N$, we can transport everything to $E_{10}$ which we identify with $S_Y$ via the marking. Then $\stab(\overline{\Gamma}_0, [\omega]) \cong \Gbar_Y$ and $\overline{\Gamma}_h \cong \overline{O(E_{10},h)}$.
Let $g_1 \in O(E_{10})$ such that $\tilde{g} = (g_1,g_2) \in G$ and set $h' = h{g_1^{-1}}$.
Then $(h',\omega)\tilde g = (h, \omega g_2)$. Now $\overline{g}_2^{-1} \overline{\Gamma}_h \overline{g}_2 \cong \overline{g}_1^{-1} \overline{O(E_{10},h)} \overline{g}_1 = \overline{O(E_{10},h')}= O(S_Y, h')$.
\end{proof}
\begin{remark}
In \cite[Table 8.3]{enriquesII} Dolgachev and Kondo list the degrees, orbits and numbers $r_i'$ for general 1-nodal Enriques surfaces. They remark that the weighted sum of the orbits with weights $r_i'$ adds up to the total degree of the map. They speculate that the $r_i'$ should be the ramification degrees $r_i$ of $\M_{En,h} \to \M_{En}$. \Cref{thm:branch} confirms this.
\end{remark}

\subsection{Weyl polarizations}
Recall that the projectivization of the positive cone $\P_Y$ is a $9$-dimensional hyperbolic space $\P_Y/\RR_+ = \mathbb{H}^9$ which comes with a hyperbolic metric and a volume. The hyperbolic volume of a cone $C$ in $\P_Y$ is defined as the volume of its projectivization $C/\RR_+ \subseteq \mathbb{H}^9$.
\begin{theorem}
Let $(Y,W)$ be a numerically Weyl quasi-polarized complex Enriques surface.
\begin{enumerate}
\item The ramification degree of $[(Y,W)]$ under the map $\P^v_0 \colon \M_{En,v} \to \M_{En,0}$
is given by $\sharp \Gbar_Y$.
\item The hyperbolic volume of the fundamental domain of $\aut_Y$ on $\Nef(Y)$ is given by
\[ \#(\P^v_0)^{-1}([Y])1_{\mathrm{Vin}}= \left[ O(S_Y \otimes \FF_2) : \Gbar_Y\right]1_{\mathrm{Vin}}= 1_{\mathrm{BP}}/\#\Gbar_Y\]
where $1_{\mathrm{Vin}}$ is the hyperbolic volume of a Vinberg chamber
and $1_{\mathrm{BP}}$ the hyperbolic volume of $\Nef(Y_0)/\aut(Y_0)$ for $Y_0$ a general Enriques surface.
\end{enumerate}
\end{theorem}
\begin{proof}
For the Weyl vector $v$ we have $O(E_{10},v) = 1$ and therefore the formulas simplify.
1) Follows directly from \Cref{thm:branch}.
2) The fundamental domain of $\aut(Y)$ on $\Nef(Y)$ tesselated by Vinberg chambers.
Their number is given by the cardinality of the fiber of $\P^v_0$ over $[Y]$ which is $[O(E_{10}\otimes \FF_2): \Gbar_Y]$.
Finally, for a general Enrique surface $\Gbar_{Y_0}=1$ is trivial; hence $1_\mathrm{{BP}}= \# O(S_Y\otimes \FF_2) 1_{\mathrm{Vin}}$.
\end{proof}

\begin{remark}\label{rmk:modular_w_h}
Let $h \in E_{10}$ with $h^2>0$ and $v \in E_{10}$ a fixed Weyl vector. We can give a modular interpretation of the map
\[\P^v_h \colon \M_{En,v} \to \M_{En, h} \]
as follows.
Recall from \Cref{sec:vinberg} that the Weyl vector $v$ corresponds to a fundamental root system $\{e_1,\dots e_{10}\}=\{e \in E_{10} \mid e^2 = -2, e.v =1\}$, which forms an $E_{10}$ diagram.

Let $(Y,W)$ be a Weyl polarized Enriques surface. Then \[\{e \in S_Y \mid e^2 = -2, e.W=1\} = \{e_1',\dots e_{10}'\}\]
is an $E_{10}$ diagram and the labels are uniquely determined as in \Cref{figure:E10}. Thus $e_i' \mapsto e_i$ defines an isomorphism $\phi_W \colon S_Y \to E_{10}$.
It satisfies $\phi_W(W) = v$. 

Choose and fix $h'$ in $O(E_{10})h\cap C(v)$ and map
$(Y,W) \mapsto (Y,H)$ with $H:=\Phi_W^{-1}(h')$. Since $H \in \phi_W^{-1}(C(v))$ is contained in the nef cone of $Y$, $(Y,H)$ is indeed numerically $h$-quasi-polarized.
If we choose another $h'' \in O(E_{10})h \cap C(w)$, then we get a different map. This is in accordance with the maps $\P^v_h$ depending on the choice of $h$ and different choices lead to different (but conjugate) groups $\Gamma_{h'}$.

\end{remark}
\begin{remark}
The previous construction works in families and in positive characteristic as well. Hence it should give a morphism of the moduli stacks of numerically quasi-polarized Enriques surfaces similar to the situation over $\CC$:
\[\M_{En,v}\to \M_{En,h} \to \M_{En,0}.\]
It is not clear to us whether the first map $\P^v_h$ can be interpreted as Galois covers. On the level of isomorphism classes, one can define a \emph{pointwise} group action of $S=O(E_{10} \otimes \FF_2)$ on $\M_{En,w}$ such that $[Y,W]$ and $[Y',W']$ are in the same $S$ orbit if and only if $[Y]=[Y']$. However this does not seem to work in families because our construction of the group action involves the Weyl group $W(Y)$ of $Y$, which can vary wildly in families.
\end{remark}

\bibliographystyle{alpha}
\bibliography{527elliptic_fibrations}

\section{Appendix: Tables}
The following tables give examples for the ramification indices $r_h$ as in \Cref{thm:branch}.
The data was generated with the help of the computer algebra system OSCAR \cite{OSCAR,OSCAR-book}.

There are 1,2,2,2,3 $O(E_{10})$-orbits of primitive vectors $h\in E_{10}$ with $h^2 = 2,4,6,8,10$. For fixed $h^2$ they are distinguished by their $\phi$-invariant
$\phi(h):=\min \{h.f \mid 0\neq f\in E_{10}, f^2=0\}$. See \Cref{table1} 

\begin{table}
\caption{$O(E_{10})$ - orbits of vectors}\label{table1}
\begin{tabular}{ccc}
\toprule
\(h^2\) & \(\phi(h)\) & \([O(E_{10}\otimes \FF_2): \overline{O(E_{10},h})]\)\\
\hline
0 & - &\( 17 \cdot 31\)\\
2 & 1 & \(2^7 \cdot 17  \cdot 31\)\\
4 & 1 & \(2^8\cdot 17\cdot 31\)\\
4 & 2 & \(2^6\cdot 3\cdot 5\cdot 17\cdot 31\)\\
6 & 1 & \(2^8\cdot 17\cdot 31\)\\
6 & 2 & \(2^{10}\cdot 5\cdot 17\cdot 31\)\\
8 & 1 & \(2^8\cdot 17\cdot 31\)\\
8 & 2 & \(2^7\cdot 3^3\cdot 5\cdot 17\cdot 31\)\\
10 & 1 & \(2^8\cdot 17\cdot 31\)\\
10 & 2 & \(2^{10}\cdot 3\cdot 5\cdot 17\cdot 31\)\\
10 & 3 & \(2^{13}\cdot 3\cdot 17\cdot 31\)\\
\bottomrule
\end{tabular}
\end{table}

\begin{table}
\caption{$\Aut(Y)$-orbits on $(A_1,A_1)$-generic Enriques surfaces}
\renewcommand{\arraystretch}{1.3}
 \rowcolors{1}{}{lightgray}
 \begin{tabular}[t]{ccccc}
 \toprule
$h^2$ &$\phi(h)$&  singularities & number of $\Aut(Y)$-orbits & $r_h$ \\
\hline
2 &1 & $\emptyset$ &16320 & 2 \\
& & $A_1$ & 16456 & 1 \\
& & $A_1$ &9180& 2 \\
4 &2 & $\emptyset$ &130560 & 2 \\
& & $A_1$ & 73440 & 1 \\
& & $A_1$ &85680& 2 \\
 6 & $2$ & $\emptyset$ & 652800& 2 \\
& & $A_1$ &364480 & 1 \\
& & $A_1$ &514080 & 2 \\
 8 & $2$ & $\emptyset$ & 2350080 & 2 \\
& & $A_1$ & 1028160& 1 \\
&  & $A_1$ &  1689120 & 2 \\
 10 & $3$ & $\emptyset$ & 3133440& 2 \\
& & $A_1$ & 1175040 & 1 \\
& & $A_1$ & 2754816 & 2 \\
\bottomrule
\end{tabular}
\end{table}

\renewcommand{\arraystretch}{1.3}
 \rowcolors{1}{}{lightgray}
 \begin{longtable}{ccccc}
 \caption{$\Aut(Y)$-orbits on $(E_8,E_8)$-generic Enriques surfaces}\\
\toprule
$h^2$ &$\phi(h)$&  singularities & number of $\Aut(Y)$-orbits & $r_h$\\
\hline
  \endfirsthead
 \toprule
$h^2$ &$\phi(h)$&  singularities & number of $\Aut(Y)$-orbits & $r_h$ \\
\hline
\endhead
 \endfoot
 \endlastfoot
2 &1 & \(D_{8} \) &2 & 135 \\
 & & \(D_{8} \) &1 & 8640 \\
 & & \(D_{8} \) &1 & 9450 \\
 & & \(E_{7} A_{1} \) &2 & 120 \\
 & & \(E_{7} A_{1} \) &1 & 3360 \\
 & & \(E_{7} A_{1} \) &3 & 4320 \\
 & & \(E_{7} A_{1} \) &2 & 7560 \\
 & & \(E_{7} A_{1} \) &2 & 8640 \\
 & & \(E_{8} \) &1 & 1 \\
 & & \(E_{8} \) &1 & 135 \\
4 & 2& \(D_{8} \) &1 & 1080 \\
 & & \(D_{8} \) &2 & 7560 \\
 & & \(D_{8} \) &2 & 8640 \\
 & & \(D_{8} \) &1 & 37800 \\
 & & \(D_{8} \) &1 & 75600 \\
 & & \(D_{8} \) &2 & 120960 \\
 & & \(E_{7} A_{1} \) &1 & 34560 \\
 & & \(E_{7} A_{1} \) &1 & 80640 \\
 & & \(E_{8} \) &2 & 960 \\
6 &2 & \(D_{8} \) &1 & 8640 \\
 & & \(D_{8} \) &1 & 30240 \\
 & & \(D_{8} \) &2 & 120960 \\
 & & \(D_{8} \) &1 & 151200 \\
 & & \(D_{8} \) &1 & 302400 \\
 & & \(E_{7} A_{1} \) &1 & 120 \\
 & & \(E_{7} A_{1} \) &2 & 3360 \\
 & & \(E_{7} A_{1} \) &2 & 4320 \\
 & & \(E_{7} A_{1} \) &3 & 7560 \\
 & & \(E_{7} A_{1} \) &3 & 40320 \\
 & & \(E_{7} A_{1} \) &2 & 90720 \\
 & & \(E_{7} A_{1} \) &5 & 120960 \\
 & & \(E_{7} A_{1} \) &2 & 151200 \\
 & & \(E_{7} A_{1} \) &1 & 226800 \\
 & & \(E_{7} A_{1} \) &2 & 241920 \\
 & & \(E_{8} \) &1 & 1120 \\
 & & \(E_{8} \) &1 & 4320 \\
8 & 2& \(D_{8} \) &2 & 1080 \\
 & & \(D_{8} \) &2 & 7560 \\
 & & \(D_{8} \) &6 & 8640 \\
 & & \(D_{8} \) &4 & 69120 \\
 & & \(D_{8} \) &4 & 75600 \\
 & & \(D_{8} \) &2 & 241920 \\
 & & \(D_{8} \) &2 & 302400 \\
 & & \(D_{8} \) &2 & 453600 \\
 & & \(D_{8} \) &4 & 483840 \\
 & & \(D_{8} \) &4 & 604800 \\
 & & \(E_{7} A_{1} \) &1 & 15120 \\
 & & \(E_{7} A_{1} \) &2 & 45360 \\
 & & \(E_{7} A_{1} \) &2 & 69120 \\
 & & \(E_{7} A_{1} \) &1 & 151200 \\
 & & \(E_{7} A_{1} \) &1 & 226800 \\
 & & \(E_{7} A_{1} \) &2 & 241920 \\
 & & \(E_{7} A_{1} \) &2 & 483840 \\
 & & \(E_{8} \) &2 & 1080 \\
 & & \(E_{8} \) &2 & 7560 \\
 & & \(E_{8} \) &2 & 8640 \\
10 & 3& \(D_{8} \) &2 & 8640 \\
 & & \(D_{8} \) &4 & 120960 \\
 & & \(D_{8} \) &1 & 604800 \\
 & & \(D_{8} \) &1 & 2419200 \\
 & & \(E_{7} A_{1} \) &4 & 34560 \\
 & & \(E_{7} A_{1} \) &1 & 69120 \\
 & & \(E_{7} A_{1} \) &4 & 80640 \\
 & & \(E_{7} A_{1} \) &4 & 241920 \\
 & & \(E_{7} A_{1} \) &3 & 483840 \\
 & & \(E_{7} A_{1} \) &4 & 1209600 \\
 & & \(E_{7} A_{1} \) &1 & 1612800 \\
 & & \(E_{8} \) &2 & 960 \\
 & & \(E_{8} \) &1 & 24192 \\
\bottomrule
\end{longtable}

\renewcommand{\arraystretch}{1.3}
 \rowcolors{1}{}{lightgray}
 \begin{longtable}{ccccc}
 \caption{$\Aut(Y)$-orbits on $(A_2,A_2)$-generic Enriques surfaces}\\
\toprule
$h^2$ &$\phi(h)$&  singularities & number of $\Aut(Y)$-orbits & $r_h$\\
\hline
  \endfirsthead
 \toprule
$h^2$ &$\phi(h)$&  singularities & number of $\Aut(Y)$-orbits & $r_h$ \\
\hline
\endhead
 \endfoot
 \endlastfoot
2 & & \(A_{1} \) &7752 & 3 \\
 & & \(A_{1} \) &4284 & 6 \\
 & & \(A_{2} \) &3808 & 1 \\
 & & \(A_{2} \) &4896 & 3 \\
4 & & \(A_{1} \) &39168 & 3 \\
 & & \(A_{1} \) &45696 & 6 \\
 & & \(A_{2} \) &8568 & 1 \\
 & & \(A_{2} \) &25704 & 3 \\
 & & \(A_{2} \) &4760 & 6 \\
6 & & \(A_{1} \) &172992 & 3 \\
 & & \(A_{1} \) &239904 & 6 \\
 & & \(A_{2} \) &45832 & 1 \\
 & & \(A_{2} \) &145656 & 3 \\
 & & \(A_{2} \) &42840 & 6 \\
8 & & \(A_{1} \) &548352 & 3 \\
 & & \(A_{1} \) &900864 & 6 \\
 & & \(A_{2} \) &102816 & 1 \\
 & & \(A_{2} \) &376992 & 3 \\
 & & \(A_{2} \) &137088 & 6 \\
10 & & \(A_{1} \) &548352 & 3 \\
 & & \(A_{1} \) &1292544 & 6 \\
 & & \(A_{2} \) &78336 & 1 \\
 & & \(A_{2} \) &548352 & 3 \\
 & & \(A_{2} \) &304640 & 6 \\
 \bottomrule
\end{longtable}

\renewcommand{\arraystretch}{1.3}
 \rowcolors{1}{}{lightgray}
 \begin{longtable}{ccccc}
 \caption{$\Aut(Y)$-orbits on $(3A_1,3A_1)$-generic Enriques surfaces}\\
\toprule
$h^2$ &$\phi(h)$&  singularities & number of $\Aut(Y)$-orbits & $r_h$\\
\hline
  \endfirsthead
 \toprule
$h^2$ &$\phi(h)$&  singularities & number of $\Aut(Y)$-orbits & $r_h$ \\
\hline
\endhead
 \endfoot
 \endlastfoot
2 & & \(\) &816 & 8 \\
 & & \(A_{1} \) &3072 & 4 \\
 & & \(A_{1} \) &1728 & 8 \\
 & & \(A_{1} A_{1} \) &3072 & 2 \\
 & & \(A_{1} A_{1} \) &3072 & 4 \\
 & & \(A_{1} A_{1} \) &960 & 8 \\
 & & \(A_{1} A_{1} A_{1} \) &1056 & 1 \\
 & & \(A_{1} A_{1} A_{1} \) &1584 & 2 \\
 & & \(A_{1} A_{1} A_{1} \) &760 & 4 \\
 & & \(A_{1} A_{1} A_{1} \) &180 & 8 \\
4 & & \(\) &9600 & 8 \\
 & & \(A_{1} \) &13824 & 4 \\
 & & \(A_{1} \) &16128 & 8 \\
 & & \(A_{1} A_{1} \) &9216 & 2 \\
 & & \(A_{1} A_{1} \) &18432 & 4 \\
 & & \(A_{1} A_{1} \) &11520 & 8 \\
 & & \(A_{1} A_{1} A_{1} \) &2208 & 1 \\
 & & \(A_{1} A_{1} A_{1} \) &6336 & 2 \\
 & & \(A_{1} A_{1} A_{1} \) &6120 & 4 \\
 & & \(A_{1} A_{1} A_{1} \) &2640 & 8 \\
6 & & \(\) &32640 & 8 \\
 & & \(A_{1} \) &67584 & 4 \\
 & & \(A_{1} \) &96768 & 8 \\
 & & \(A_{1} A_{1} \) &36864 & 2 \\
 & & \(A_{1} A_{1} \) &98304 & 4 \\
 & & \(A_{1} A_{1} \) &72192 & 8 \\
 & & \(A_{1} A_{1} A_{1} \) &8192 & 1 \\
 & & \(A_{1} A_{1} A_{1} \) &27648 & 2 \\
 & & \(A_{1} A_{1} A_{1} \) &37056 & 4 \\
 & & \(A_{1} A_{1} A_{1} \) &17056 & 8 \\
8 & & \(\) &172800 & 8 \\
 & & \(A_{1} \) &193536 & 4 \\
 & & \(A_{1} \) &317952 & 8 \\
 & & \(A_{1} A_{1} \) &95232 & 2 \\
 & & \(A_{1} A_{1} \) &293376 & 4 \\
 & & \(A_{1} A_{1} \) &244224 & 8 \\
 & & \(A_{1} A_{1} A_{1} \) &14400 & 1 \\
 & & \(A_{1} A_{1} A_{1} \) &72192 & 2 \\
 & & \(A_{1} A_{1} A_{1} \) &109008 & 4 \\
 & & \(A_{1} A_{1} A_{1} \) &61728 & 8 \\
10 & & \(\) &156672 & 8 \\
 & & \(A_{1} \) &221184 & 4 \\
 & & \(A_{1} \) &516096 & 8 \\
 & & \(A_{1} A_{1} \) &98304 & 2 \\
 & & \(A_{1} A_{1} \) &344064 & 4 \\
 & & \(A_{1} A_{1} \) &430080 & 8 \\
 & & \(A_{1} A_{1} A_{1} \) &15360 & 1 \\
 & & \(A_{1} A_{1} A_{1} \) &72192 & 2 \\
 & & \(A_{1} A_{1} A_{1} \) &134016 & 4 \\
 & & \(A_{1} A_{1} A_{1} \) &121920 & 8 \\
 \hline
\end{longtable}

\end{document}